\documentclass[12pt,a4paper]{amsart}

\usepackage{amssymb,graphicx}

\def\Q{\mathbb{Q}}
\def\Z{\mathbb{Z}}
\def\N{\mathbb{N}}

\def\cC{{\mathcal C}}

\def\G{\Gamma}
\def\g{\gamma}

\def\a{\alpha}
\def\b{\beta}
\def\L{\Lambda}
\def\FP{{\rm FP}}
\def\FR{{\rm FR}}
\def\F{{\rm F}}
\def\<{\langle}
\def\>{\rangle}
\def\sg#1{\langle { #1} \rangle}

\def\isom{\cong}
\def\nc#1{\langle\langle { #1} \rangle\rangle}

\def\ab#1{#1_{\rm ab}}
\def\go#1{H_1(#1,\Z)/{\rm (torsion)}}
\def\ee#1{\exists{{\rm{Env}}(#1)}}
\def\eer#1{\exists{{\rm{Env}}_0(#1)}}
\def\Nil{{\rm{Nilp}}_\exists}

\def\P{\mathcal{P}}
\def\ssm{\smallsetminus}

\def\ch{\binom}

\newtheorem{letterthm}{Theorem}

\newtheorem{lettercly}[letterthm]{Corollary}

\newtheorem{theorem}{Theorem}[section]
\newtheorem{thm}[theorem]{Theorem}
\newtheorem{lemma}[theorem]{Lemma}
\newtheorem{prop}[theorem]{Proposition}
\newtheorem{proposition}[theorem]{Proposition}
\newtheorem{addendum}[theorem]{Addendum}

\newtheorem{corollary}[theorem]{Corollary}

\theoremstyle{definition}
\newtheorem{remark}[theorem]{Remark} 
\newtheorem{remarks}[theorem]{Remarks} 
\newtheorem{definition}[theorem]{Definition} 
\newtheorem{example}[theorem]{Example} 

\newenvironment{pf}{\par\medskip\noindent\textit{Proof.}~}{\hfill $\square$\par\medskip}

\begin{document}

\catcode`\@=11
\def\serieslogo@{\relax}
\def\@setcopyright{\relax}
\catcode`\@=12
 
\title[Finite presentation and residually free groups]{On the finite presentation of subdirect products and 
the nature of residually free groups}

\author[Bridson]{Martin R.~Bridson}
\address{Martin R.~Bridson\\
Mathematical Institute \\
24--29 St Giles'\\
Oxford OX1 3LB  \\ 
U.K. }
\email{bridson@maths.ox.ac.uk}

\author[Howie]{James Howie}
\address{ James Howie\\
Department of Mathematics and the Maxwell Institute for Mathematical Sciences\\
Heriot--Watt University\\
Edinburgh EH14 4AS }
\email{ jim@ma.hw.ac.uk}

\author[Miller]{ Charles~F. Miller~III }
\address{ Charles~F. Miller~III\\
Department of Mathematics and Statistics\\
University of Melbourne\\
Melbourne 3010, Australia }
\email{ c.miller@ms.unimelb.edu.au }

\author[Short]{ Hamish Short }
\address{ Hamish Short \\
L.A.T.P., U.M.R. 7353\\
Centre de Math\'ematiques et d'Informatique\\
39 Rue Joliot--Curie\\
Universit\'e d'Aix-Marseille, F--13453\\
Marseille cedex 13, France }
\email{ hamish@cmi.univ-mrs.fr }
 
\thanks{This work grew out of a project funded by  l'Alliance Scientific grant \# PN 05.004.
Bridson was supported in part by an EPSRC Senior Fellowship
and a Royal Society Nuffield Research Merit Award.  
Howie was supported in part by Leverhulme Trust grant F/00 276/J}




\begin{abstract} 
We establish {\em{virtual surjection to pairs}} (VSP) as
 a general criterion for the finite presentability of 
subdirect products of groups: if $\G_1,\dots,\G_n$ are finitely
presented and $S<\G_1\times\dots\times\G_n$ projects to a 
subgroup of finite index in each $\G_i\times\G_j$, then $S$ is 
finitely presentable, indeed there is an algorithm that will construct a finite
presentation for $S$.

We use the VSP criterion to characterise the
finitely presented residually free groups.
We prove that the
class of such groups is recursively enumerable.
We describe an algorithm that,
given a finite presentation of a residually free group,
constructs a canonical embedding into a direct product of finitely many limit groups.
We solve
the (multiple) conjugacy problem and membership problem for finitely presentable subgroups
of residually free groups. We also prove that there is an algorithm that, given a finite generating
set for such a subgroup, will construct a finite presentation.

New families of subdirect
products of free groups are constructed, including the first examples
of finitely presented subgroups that are neither  ${\rm{FP}}_\infty$ 
nor of Stallings-Bieri type.   
 \end{abstract}

\maketitle

\section{Introduction}

A very challenging problem is to determine 
which subgroups $S< G_1\times\dots\times G_n$ of a
direct product of finitely presented groups are themselves finitely presented. Indeed this
problem is subtle even when the $G_i$ are free groups. 

Some terminology is useful for describing how a subgroup 
sits inside a direct product. A subgroup of a direct
product of groups is termed a {\em subdirect product} if its projection to 
each factor is surjective.  A subdirect product is said to be {\em{full}} 
if it intersects each of the direct factors non-trivially.  A subgroup 
$S<G_1\times\dots\times G_n$ 
 is said to be {\em virtually surjective on pairs}  (VSP)
if for all $i\neq j\in\{1,\dots,n\}$, the
projection $p_{ij}(S)\subset G_i\times G_j$ has finite index. 
(We implicitly assume that $n\ge2$.)

In  \cite{BM} we showed that the full subdirect products 
of non-abelian free and surface groups which are finitely presented must satisfy the VSP condition, from
which it follows that they contain a term of the lower central series
of a subgroup of finite index in the direct product.  The VSP condition also played an important role 
in our previous work on subdirect products of limit groups \cite{BHMS1}.
The first purpose of this article is to establish VSP as a 
criterion for the finite presentability of subgroups of more general direct products of groups. 
We remind the reader that a subgroup $S<\G$ is termed {\em separable} if for every 
$\gamma\in\G\ssm S$ there exists a normal subgroup  $K\triangleleft\G$ of finite index such that  
$\gamma\notin SK$.

\begin{letterthm}[The VSP Criterion]\label{t:pairs} 
Let $S<G_1\times\dots\times G_n$ be a
subgroup of a direct product of finitely presented groups. If 
$S$ is virtually surjective on pairs (VSP),
then it is finitely presented and separable.
\end{letterthm}

Note that we do not assume,{\em  a priori}, that the subgroup $S$ is
finitely generated.
The converse of Theorem \ref{t:pairs} is false in general; even
finitely presented full subdirect products need not satisfy VSP.  For example, if
$N$ is a finitely-generated torsion-free nilpotent group that is not  
cyclic, and if  $\phi: N\times N\to\Z$ is a homomorphism whose restriction to each factor is non-trivial,
then the kernel of $\phi$ is a finitely presented, separable full subdirect product without VSP. 

An essential ingredient in the proof of Theorem \ref{t:pairs} is  the following
asymmetric version of the 1-2-3 Theorem of \cite{BBMS}.

\begin{letterthm}[Asymmetric 1-2-3 Theorem]\label{t:123}
Let $f_1:\G_1\to Q$ and $f_2:\G_2\to Q$ be surjective group homomorphisms. Suppose
that $\G_1$ and $\G_2$ are finitely presented, that  $Q$ is
of type ${\rm{F}}_3$, and that
at least one of $\ker f_1 $ and $\ker f_2 $ is finitely generated. Then
the fibre product of $f_1$ and $f_2$,
$$ P = \{(g,h) \mid f_1(g)=f_2(h)\}\subset\G_1\times\G_2,$$
is finitely presented.
\end{letterthm}

We shall concentrate on the {\em effective} version of this result (Theorem \ref{t:123eff}) which yields an explicit finite presentation
for $P$. (Proofs of the non-effective version can
be found in \cite{BHMS2} and \cite{will}.)
In Theorem  \ref{t:effFPsubdirect} we use Theorem \ref{t:123eff} to prove an effective version  of Theorem \ref{t:pairs}:
there is a uniform partial algorithm that,
given finite presentations for the factors $G_i$ and a finite generating set for $S$ satisfying VSP, will output a finite presentation
for $S$. 

In this paper we describe a number of algorithmic processes. 
Often they are {\em partial algorithms} meaning that when applied 
to an object $X$ which satisfies some condition $\cC$, the process 
will halt with some appropriate information about $X$;  but if $X$ does 
not satisfy $\cC$ either the process will halt saying $X\notin \cC$ or 
the process will fail to halt.  

For instance in Theorem  \ref{t:effFPsubdirect},  which is our  effective version of Theorem \ref{t:pairs},
we describe such a partial algorithm which, when given a direct product of finitely presented groups
and a finite set of generators for a 
subgroup $S$,  succeeds  when $S$ actually satisfies VSP and yields a finite presentation for $S$.  
The algorithm is {\em uniform} in the given data (a direct product and a finite generating set 
 for $S$) and so
can be started without knowing whether or not $S$ satisfies VSP.  If $S$ does not satisfy VSP, the
algorithm of Theorem  \ref{t:effFPsubdirect} does not halt.  
(In a direct product of free groups for instance, there is no algorithm to determine whether or not the subgroup 
generated by a finite set satisfies VSP nor whether it is finitely presentable.)

\smallskip

\subsection*{Residually free groups}

\smallskip

In the second half of this article we use Theorem \ref{t:pairs} and more
specialised results 
to advance the understanding of finitely presented residually free groups.
Residually free groups provide a context for 
a rich and powerful interplay among group theory, topology and logic.
By definition, a group $G$ is {\em residually free} if, for every
$1\ne g\in G$, there is a homomorphism $\phi$ from $G$ to a free group
$F$ such that $1\ne\phi(g)$ in $F$. The prototypes for these groups are the
finitely presented subgroups of finite direct products of free and surface
groups. In general a finitely-presented residually-free group is a full subdirect
product of finitely many limit groups, i.e.
it can be embedded in 
a finite direct product of limit groups so that it intersects each factor
non-trivially and  projects onto each factor (cf.~Theorem \ref{t:ee(S)} below).  
In our earlier studies \cite{BM}, \cite{BH2}, \cite{BHMS},  \cite{BHMS1},
we proved that these full subdirect products
have finite index in the ambient 
product if they are of  type $\FP_\infty$.
We also proved that in general they
virtually contain a term of the lower central series of the product.
These tight restrictions set the {\em finitely presented} subdirect
products of limit groups apart from those that are merely
{\em{finitely generated}}, since
the finitely generated
subgroups of the direct product of two free groups
are already  hopelessly complicated \cite{cfm-thesis}.
Nevertheless,  a
thorough understanding of the 
finitely presented subdirect products of free
and limit groups has remained a distant
prospect, with only a few types of examples known.

In this article we pursue such an understanding in a number of ways.
Using Theorem \ref{t:pairs}, we characterize  finitely-presented residually-free groups 
among the full subdirect products of limit groups in terms of their projections
to the direct factors.   A revealing family of finitely 
presented full subdirect products of free groups is constructed; this gives rise
to a more constructive characterization of finitely presented residually free groups.
We give algorithms for finding finite presentations when they exist, for constructing certain canonical
embeddings, for enumerating finitely presented residually free groups, and for
solving their conjugacy and membership problems.

\smallskip
By definition, a group $G$ is residually free if it is
isomorphic to a subgroup of an unrestricted direct product of free groups. 
In general, one requires infinitely many factors in this direct
product, even if $G$ is finitely generated. For example, the 
fundamental group of a closed orientable surface $\Sigma$
is residually free but
it cannot be embedded in a finite direct product
if $\chi(\Sigma)<0$, since $\pi_1\Sigma$
does not contain $\Z^2$ and is not a subgroup of a free group.
However,  
Baumslag, Myasnikov and Remeslennikov \cite[Corollary 19]{BMR} 
proved that one can force the enveloping product to be finite
at the cost of replacing free groups by 
{\em{$\exists$-free groups}} (see also \cite[Corollary 2]{KM2} and \cite[Claim 7.5]{Se1}).  
In \cite{KMeffective} Kharlampovich and Myasnikov describe an algorithm
to find such an embedding, based on the deep work of Makanin \cite{Mak}
and Razborov \cite{R}. We shall describe a new algorithm that
does not depend on \cite{Mak} and \cite{R}; the embedding that
we construct is canonical in a strong sense (see Theorem \ref{t:ee(S)}).

By definition, 
$\exists$-free groups have the same universal theory as a free group; 
they are now more commonly known as
{\em{limit groups}}, a term coined by Sela \cite{Se1}.
They have been
much studied in recent years in connection with
Tarski's problems on the first order
 logic of free groups \cite{Se1}, \cite{KM2}. They have
been shown to enjoy a rich geometric structure. A useful characterisation of
limit groups is that they are the  finitely generated groups $G$
that  are  {\em fully residually free}:
for every finite subset $A\subset G$, there is a homomorphism from $G$ to a free
group that restricts to an injection on $A$.  

For the most part, we
treat finitely generated residually free groups $S$ as subdirect products of limit groups.
There are at least two obvious drawbacks to this approach: the ambient
product of limit groups is not canonically associated to  $S$;
and given a direct product of limit groups, it is difficult to
determine which finitely generated subgroups are finitely presented.

The first of these drawbacks is overcome by items (1), (3) 
and (4) of 
the following theorem. Item (2) is based on 
Theorem 4.2 of  \cite{BHMS1}.   

\begin{letterthm}\label{t:ee(S)} There is an algorithm that, given
a finite presentation of a residually free group $S$, will construct
an embedding  $\iota:S\hookrightarrow \ee{S}$, so that
\begin{enumerate}
\item $\ee{S} = \G_{\rm ab} \times \eer{S}$  where
$\G_{\rm ab} = \go{S}$ and $\eer{S} = \G_1\times\dots\times\G_n$
is a direct product of non-abelian limit groups $\G_i$.
The intersection of $S$ with the 
kernel of the projection $\rho:\ee{S}\to\eer{S}$ 
is the centre $Z(S)$ of $S$, and $\rho(S)$
is a full subdirect product.
\item
$L_i:=\G_i\cap S$ contains a term of the lower central series of
a subgroup of finite index in $\G_i$, for $i=1,\dots,n$, and
therefore $\Nil(S):=\ee{S}/(L_1\times\cdots\times L_n)$ 
is virtually nilpotent.
\item {\rm [Universal Property]}  For every homomorphism 
$\phi:S
\to D=\Lambda_1\times\dots\times\Lambda_m$, 
with $\phi(S)$ subdirect and  the $\Lambda_i$
non-abelian limit groups, there exists a unique homomorphism 
$\hat\phi : \eer{S}\to D$ with $\hat\phi\circ\rho|_S = \phi$;
\item {\rm [Uniqueness]}  moreover,
if 
$\phi: S\hookrightarrow D$ 
embeds $S$ as a  full subdirect  product, then  $\hat\phi:\eer{S}\to D$ is an isomorphism
that respects the direct sum decomposition.
\end{enumerate}
\end{letterthm}

The group $\ee{S}$ in Theorem \ref{t:ee(S)}  is called the {\em existential 
envelope} of $S$ and the associated factor $\eer{S}$ is the
{\em reduced existential envelope}. 
The projection $\rho$ embeds  $S/Z(S)$ in $\eer{S}$, and $\rho(S)\subset\eer{S}$ is always a full subdirect product.  The subgroup $S\subset \ee{S}$ 
is always a subdirect product but it is
full if and only if $S$ has a non-trivial centre.

Proceeding in the opposite direction, 
Guirardel and Levitt \cite{GL} prove that, given a subdirect
product $S$ of limit groups, one can algorithmically construct a finitely
presented group whose maximal centreless residually free quotient is
isomorphic to $S/Z(S)$. They also show that there is no algorithm
to determine whether the maximal residually free quotient of a finitely
presentable group is finitely presentable.  In a similar vein, we note  that
there is no algorithm to determine whether or not a finitely generated subdirect
product of limit groups is finitely presentable. Indeed, if $F$ is a non-abelian
free group, then there is no algorithm to  determine which finitely generated
full subdirect products of $F\times F$ are finitely presentable (cf.~\cite{BM}).

The second of the drawbacks we identified 
in the discussion preceding Theorem \ref{t:ee(S)} is
resolved by item (4) of the following theorem.
In order to state this theorem
concisely we introduce the following temporary definition: an embedding $S\hookrightarrow
\G_0\times\dots\times\G_n$ of a residually free group $S$
as a full subdirect product of  limit groups is said to be {\em{neat}}
if $\G_0$ is
abelian (possibly trivial), $S\cap\Gamma_0$ is of finite index in $\G_0$, and
$\G_i$  is non-abelian for $i=1,\dots,n$.

\begin{letterthm}\label{t:main} Let $S$ be a  finitely generated
residually free group.  Then the  following conditions are equivalent:
\begin{enumerate}
\item $S$ is finitely presentable;
\item $S$ is of type $\mathrm{FP}_2(\Q)$;
\item ${\rm{dim}\,}H_2(S_0;\Q)<\infty$ for all subgroups $S_0\subset S$ of
finite index;
\item {\rm [$\exists$ neat VSP]} there exists a neat embedding  $S\hookrightarrow
\G_0\times\dots\times\G_n$ into a product of limit groups 
such that
the image of $S$ under the projection to $\G_i\times\G_j$ has finite index
 for $1\le i<j\le n$;
\item {\rm [neat $\implies$ VSP]} for every neat embedding $S\hookrightarrow
\Lambda_0\times\dots\times\Lambda_n$ into a product of limit groups,
the image of $S$ under the projection to $\Lambda_i\times\Lambda_j$ has finite index
 for $1\le i<j\le n$.
\end{enumerate}
\end{letterthm}
   
\begin{lettercly}\label{c:F_n=FP_n} For all $k\in\N$, a residually
 free group $S$ is of type ${\rm{F}_k}$ if and only if it is of type ${\rm{FP}_k}(\Q)$.
\end{lettercly}

In this context it is worth noting that
D.~Kochloukova \cite{desi} has obtained results concerning the question
of which subdirect products of limit groups are
$\FP_k$ for $2<k <n$. In the pro-$p$ category, the analogous question has
been completely answered \cite{KS}, but for discrete groups it remains open in general.

It follows from Theorem \ref{t:main} that any subgroup $T\subset \ee{S}$ containing $S$
is again finitely presented. More generally we prove:

\begin{letterthm}\label{t:FPup} Let $k \ge 2$ be an integer,
let $S\subset D:=\G_1\times\dots\times\G_n$ be a full subdirect
 product of limit groups, and let $T\subset D$ be a subgroup that contains
$S$. If $S$ is of type ${\rm{FP}_k}(\Q)$ then so is $T$.
\end{letterthm}

The proof of Theorem \ref{t:main} relies on our earlier work 
concerning the
finiteness properties of subgroups of direct products of limit
groups \cite{BHMS1} as well as Theorem \ref{t:pairs} (the VSP criterion).

In the final section of this paper we shall combine Theorem \ref{t:main} with the effective form of
Theorem \ref{t:pairs} to prove:

\begin{letterthm}\label{t:re}
The class of finitely presented, residually free groups
is recursively enumerable. More explicitly, there exists a Turing
machine that generates a list of finite group-presentations
so that each of the groups presented is residually free and
every finitely-presented residually-free group is isomorphic to
at least one of the groups presented.
\end{letterthm}

In Section \ref{s:exs}
we turn our attention to the construction of new families of finitely-presented residually-free groups. 

Subdirect products of free groups hold a particular
historical interest, most notably in connection
with Baumslag and Roseblade's groundbreaking  work \cite{BR} on the
(non)finite presentability of  subgroups
of $F\times F$, and the seminal constructions by  Stallings \cite{St} and
Bieri \cite{Bieri} 
of finitely presentable groups that are not of type ${\rm{FP}}_\infty(\Q)$.
Subdirect products of surface groups also have a special appeal:
the work of Delzant and Gromov \cite{DeGr} shows that
such subgroups play an important role in  the problem of determining
which finitely presented groups arise as the fundamental groups of compact
K\"ahler manifolds. In the context of subdirect products
of surface groups, Dimca, Papadima and Suciu \cite{DPS} have constructed
analogues of the Bieri-Stallings examples that are fundamental groups of
smooth complex projective varieties (and hence K\"ahler).  These are currently the only known examples of subdirect products of
surface groups that are K\"ahler but not of type ${\rm{FP}}_\infty(\Q)$.

We construct  the first examples of
finitely presented subgroups of 
direct products of free groups that are neither ${\rm{FP}}_\infty(\Q)$
nor of Stallings-Bieri type, thus answering  a question raised  in \cite{BM}.
(We use the standard notation $\gamma_n(G)$ to denote the $n$-th term of 
the lower central series of a group.)

\begin{letterthm}\label{t:exs} If $c$ and $n$ are positive integers with $n\ge c+2$,
and  $D=F_1\times\dots\times F_n$ is
a direct product of free groups of rank 2, then there exists a 
finitely presented subgroup $S\subset D$ with $S\cap F_i=\gamma_{c+1}(F_i)$
for $i=1,\dots,n$. 
\end{letterthm}

$\Nil(S)$ was defined in Theorem \ref{t:ee(S)}(2).

\begin{lettercly} For all positive integers $c$ and $n\ge c+2$, there exists a 
 finitely-presented residually-free group $S$ for which
 $\Nil(S)$ is
a direct product of $n$ copies of the 2-generator free nilpotent group of class $c$.
\end{lettercly}

The proof that the  group $S$ in Theorem
\ref{t:exs} is finitely presented relies on 
our earlier structural results. Our proof of the equality
$S\cap F_i=\gamma_{c+1}(F_i)$ exploits the Magnus
embedding of the free group of rank 2 into the group
of units of 
$\Q[[\a,\b]]$, the algebra of power series in two non-commuting
variables  with rational coefficients.

Theorem \ref{t:main} describes the finitely-presented residually-free groups. 
A description of a quite different nature is given in 
Theorem \ref{t:cchar}: using a template inspired by the examples  in
Section \ref{s:exs} we prove that
every finitely-presented residually-free group is commensurable
with a particular type of subdirect product of limit groups. 

\smallskip

In Section \ref{s:decide} we apply  Theorem \ref{t:ee(S)} to elucidate the
algorithmic structure of finitely presented residually free groups. The
restriction to finitely presented groups is essential since
decision problems for arbitrary finitely generated residually
free groups are hopelessly difficult. For example,
 there are finitely generated subgroups of a direct
 product of two free groups
for which the conjugacy problem and membership problem are
unsolvable; and the isomorphism problem is unsolvable
amongst such subgroups \cite{cfm-thesis}.

The following statement includes the statement that the conjugacy problem is
solvable in every finitely-presented residually-free group.

\begin{letterthm}\label{t:conj} Let $S$ be a finitely-presented residually-free
group. There exists an algorithm that, given an integer $n$ and two
$n$-tuples of words in the generators of $S$, say $(u_1,\dots,u_n)$
and $(v_1,\dots,v_n)$, will determine whether or not there exists an
element $s\in S$ such that $su_is^{-1}=v_i$ in $S$ for $i=1,\dots,n$.
\end{letterthm} 

As previously noted, in a direct product of non-abelian free groups there is no algorithm to determine whether 
a finitely generated subgroup can be finitely presented.  Nevertheless in a finitely presented residually free group,
if we are given a finite set of generators for a subgroup which is in fact finitely presentable, then we can 
effectively find a presentation it.  The method is uniform in the given data.  Here is a more formal statement.

\begin{letterthm}\label{lt:makeFP} 
There is a uniform partial algorithm  for finding presentations of
finitely presentable subgroups of  finitely-presented residually-free groups. 
More precisely, there is a partial algorithm that, given a finite presentation for a 
residually free group $G$ and a finite set of words generating a subgroup $H$, 
will output a finite presentation for $H$ if it exists.
\end{letterthm}

In the direct product of non-abelian free groups there are finitely generated subgroups for which the
membership problem is unsolvable, but the subgroups in question are not finitely presented.
We establish a uniform solution to the membership problem for 
finitely-presentable subgroups of finitely-presented residually-free groups.

\begin{letterthm}\label{t:memb} There is a uniform partial algorithm that, given a finite presentation of a residually free group $S$,
 a finite generating set for a finitely presentable subgroup $H\subset S$ and a word $w$ in the generators of $S$, 
 can determine whether or not  $w$
 defines an element of $H$.
\end{letterthm}

Following our work, alternative approaches
 to the conjugacy and membership problems
were  developed in \cite{BWilt} and  \cite{CZ}.

In the final section of this paper we make a few remarks about the
isomorphism problem for finitely presented residually free groups,
taking account of the canonical embeddings 
$S\hookrightarrow\ee{S}$. 
 
 \smallskip
 
This paper is organised as follows. Our first goal is to prove
an effective version of the
Asymmetric 1-2-3 Theorem;
this is achieved in Section \ref{s:asym}.
In Section \ref{s:sdp} we establish 
Theorem \ref{t:pairs}. In Section \ref{s:exs} we 
construct the groups described in Theorem \ref{t:exs}.
In Section \ref{s:char} we establish the two  characterisations of
finitely presented residually free groups promised
earlier: we prove Theorem \ref{t:main} and Theorem \ref{t:cchar}.
Section \ref{s:embed} 
is devoted to the proof of Theorem \ref{t:ee(S)} and other
aspects of the canonical embedding $S\hookrightarrow\ee{S}$.
Finally, in Section \ref{s:decide}, we turn our attention 
to decidability and enumeration problems, proving Theorems
\ref{t:re}, \ref{t:conj}, \ref{lt:makeFP} and \ref{t:memb}.

Most of the results in this paper were proved in our preprint \cite{BHMS2},
posted in 2008.  In the intervening period there have been a number of related
developments by various authors \cite{BWilt,CZ,DPS,GL,desi,KS}.  In presenting the current article, we
have tried to strike a balance between the sometimes competing goals of presenting a
single coherent account of our original material, and of taking these later developments into account.

We thank   G. Baumslag,    W. Dison,
D. Kochloukova, A. Myasnikov, Z. Sela,    H. Wilton and, most particularly,
M. Vaughan-Lee for helpful comments and suggestions relating to this work.
We are also grateful to an anonymous referee for a careful reading of an earlier
version of the paper, and insightful comments that led to significant improvements.

\section{The Effective Asymmetric 1-2-3 Theorem}\label{s:asym}

In this section we prove the following effective version of Theorem \ref{t:123}. 
The basic Asymmetric 1-2-3 Theorem states that a certain type of
fibre product is finitely presented, whereas the effective version provides
an algorithm that, given natural input data, constructs a finite presentation
for the fibre product.
This enhanced version of the theorem will play a crucial role in our
proof that the class of finitely presented residually free groups
is recursively enumerable. 

In order to gain a full understanding of the proof that we are going to present, the reader
should be familiar with the original proof of the 1-2-3 Theorem as presented in \cite[\S\S 1.4--1.5]{BBMS}.
In particular, we shall not rehearse the reasons why considerations of $\pi_2\mathcal Q$ enter naturally into
the proof (cf.~remark \ref{r:needPi2}). We remind the reader that,
given a finite presentation $\mathcal Q\equiv\<X\mid R\>$ for a group $Q$, one can define
the second homotopy group of $\pi_2\mathcal Q$ to be $\pi_2$
of the standard 2-complex $K$ of the presentation, regarded as
a module over $\Z Q$ via the identification $Q=\pi_1K$. In the
present context, though, it is better to regard elements
of $\pi_2\mathcal Q$ as equivalence
classes of identity sequences $[(w_1,r_1),\dots, (w_m,r_m)]$,
where the $w_i$ are elements of the free group $F(X)$, 
the $r_i\in R^{\pm 1}$, and where $\prod_{i=1}^m w_i^{-1}r_iw_i$
is equal to the empty word in $F(X)$; equivalence is defined by Peiffer
moves, and the action of $Q$ is induced by the obvious conjugation action of
$F(X)$; see \cite{whead}.

We shall also need the following observations, which are addressed in a more pedestrian manner in \cite{BBMS}.

\begin{remark} \label{r:natPres}
Let $f:\G\to Q$ be an epimorphism where $Q=\<X\mid R\>$. Then
\begin{enumerate}
\item 
$\G$ can be presented as  $\G=\<X, C\mid \hat R,\, S\>$, so that $f$ is given by $f(x)=x$ for all $x\in X$ and
$f(c)=1$ for all $c\in C$; there is relation $\hat r\in\hat R$ of the form $r(X)=u(C)$ for 
each $r\in R$; and $S$ consists of
words that are products of subwords drawn from $C^*$, the set of conjugates of the symbols 
$c\in C$ by words in the free group on $X$. 
\item If $\ker f$ is finitely generated, then one can further assume that $C$ generates $\ker f$, and that $S$ consists of
two sets of relations: the first set expresses the fact that $\ker f$ is normal,
 with a relation of the form $x^{-\epsilon}cx^{\epsilon}=w(C)$
for each $x\in X, c\in C$ and $\epsilon = \pm 1$; the second set consists of words in the free group on $C$.
\end{enumerate}
\end{remark}

\begin{thm}\label{t:123eff} There exists a Turing machine that,
given the following data describing group homomorphisms
$f_i:\G_i\to Q\ (i=1,2)$, will output a finite presentation of
the fibre product of these maps provided that both the $f_i$
are surjective and at least one of the
kernels $\ker f_i$ is finitely generated. (If either of these
conditions fails, the machine will not halt.)

{\underline{Input Data:}}
\begin{enumerate}
\item A finite presentation $\mathcal Q\equiv \<X\mid R\>$ for $Q$.
\item A finite presentation $\<\underline a^{(i)}\mid \underline r^{(i)}\>$
for $\G_i\ (i=1,2)$.
\item $\forall a\in\underline a^{(i)}$, a word  $\hat a\in F(X)$
such that $\hat a=f_i(a)$ in $Q$.
\item A finite set of identity sequences that generates $\pi_2\mathcal Q$
as a $\Z Q$-module.
\end{enumerate}
\end{thm}

\begin{proof}
We associate to
a fixed finite group-presentation $\mathcal{Q}\equiv\< X\mid R\>$ of a group $Q$
the class $\mathcal{C}(\mathcal{Q})$ of finite group-presentations that have the 
form
$$\<X\sqcup A\sqcup B\mid S_1,S_2,S_3,S_4,S_5\>,$$
where
\begin{itemize}
\item $S_1$ consists of a relator $r(X)u_r(A)v_r(B^*)$ for each relator $r=r(X)\in R$, where
$B^*$ is the set of formal conjugates of the letters $b\in B$ by words in 
the free monoid on $X\cup X^{-1}$ and $v_r(B^*)$ is a
word in the free group on this set, while $u_r(A)$ is a word in the free group on $A$;
\item $S_2$ consists of a relator $ax^\epsilon w_{a,x,\epsilon}(A)x^{-\epsilon}$
for each $a\in A$, $x\in X$ and $\epsilon=\pm 1$, with the $w_{a,x,\epsilon}(A)$
words in the free group on $A$; 
\item $S_3=\{aba^{-1}b^{-1}\mid a\in A,b\in B\}$;
\item $S_4$ is a finite set of words in the free group on $A$;
\item $S_5$ is a finite set of words in the free group on $B^*$.
\end{itemize}

It is clear the class $\mathcal{C}(\mathcal{Q})$ is recursively enumerable.  Moreover, each group
$G$ given by a presentation $\mathcal{P}\in\mathcal{C}(\mathcal{Q})$ comes naturally
equipped with an epimorphism $\pi$ onto $Q$, namely the map that at the level
of generators restricts to the identity on $X$ and sends each element
of $A$ and $B$ to $1$.  We write
$G_A$ (resp. $G_B$) to denote the quotient of $G$ by the normal closure
$N_A$ of $A$ (resp. the normal closure $N_B$ of $B$) and regard $\pi$ as the
canonical map
$G\to G/N_AN_B\cong Q$. Note that the diagonal map $\Delta:G\to G_A\times G_B$ sends $G$ onto
the fibre-product of the epimorphisms $G_A\to Q$ and $G_B\to Q$.
Note also that $N_A$ is generated by $A$ as a group (not just a normal
subgroup) and that $N_A$ commutes with $N_B$.

Remark \ref{r:natPres}(1) assures us
that any finitely presented group $\G_1$ admitting an epimorphism $f_1:\G_1\to Q$ is isomorphic
to $G_A$ for some (indeed infinitely many) $G$ given by a presentation from
$\mathcal{C}(\mathcal{Q})$, via an isomorphism $\phi_A:G_A\to \G_1$ such that
$f_1\circ\phi_A$ is the natural epimorphism 
$$\pi_A:G_A=G/N_A\to G/N_AN_B.$$ 
Moreover, if one has a finitely presented group $\G_2$ admitting an epimorphism $f_2:\G_2\to Q$
with $\ker f_2$ finitely generated, then Remark \ref{r:natPres}(2) assures us
$\G_2$ is also isomorphic to $G_B$ for some $G$ given by a presentation in
 $\mathcal{C}(\mathcal{Q})$, via an isomorphism
$\phi_B:G_B\to \G_2$ such that $$f\circ\phi_B=\pi_B:G_B=G/N_B\to G/N_AN_B.$$ 

When one constructs a presentation in $\mathcal{C}(\mathcal{Q})$, the choice of
words involving $A$ and those involving $B$ can be made entirely independently of one another;
so the construction of maps $\phi_A$ and $\phi_B$ in the preceding paragraph can be
carried out simultaneously, i.e. using the same $G$. In other words,
for any
pair of  finitely
presented groups   $\G_1$ and $\G_2$  and epimorphisms
$f_1:\G_1\to Q$, $f_2:\G_2\to Q$ with $\ker f_2$
finitely generated, there is a group $G$ given by a presentation from $\mathcal{C}(\mathcal{Q})$
and isomorphisms $\phi_A:G_A\to \G_1$, $\phi_B:G_B\to \G_2$ such that $f_1\circ\phi_A=\pi_A$ and
$f_2\circ\phi_B=\pi_B$.

We now have sufficient notation to describe the algorithm that we seek. 
Given input data (1), (2) and (3), the algorithm
works systematically through an enumeration of $\mathcal{C}(\mathcal{Q})$, searching diagonally
for pairs $\phi_A:G_A\to \G_i$, $\phi_B:G_B\to \G_j$ as above, with $\{i,j\}=\{1,2\}$.  By
hypothesis $\G_1$ and $\G_2$ are finitely presented, and at least one of the kernels of 
$\G_k\to Q$
($k=1,2$) is finitely generated.  So our search procedure will eventually terminate successfully.

Let $\mathcal{P}\in\mathcal{C}(\mathcal{Q})$ be the presentation found by the procedure, let $G$
be the group given by $\mathcal{P}$,
and let $N_A,N_B,G_A,G_B$ be as defined above.  Then,
$(\phi_A\times\phi_B)\circ\Delta$ maps $G$ onto the fibre product $P<\G_1\times\G_2$ of $f_1$
and $f_2$, and the kernel of this map
is  $N_A\cap N_B$.  Since $N_A$ and $N_B$ commute, it follows that $N_A\cap N_B$ is central
in $N_AN_B$, and that conjugation in $G$ gives it the structure of a $\Z Q$-module.
To complete our proof, it suffices to show that we can find algorithmically 
a finite generating set $Z$
for $N_A\cap N_B$ as a $\Z Q$-module; adjoining $Z$
to the relators of $\mathcal{P}$ will then present $P$, as required.

To obtain $Z$, we follow the construction of \cite[\S\S 1.4,1.5]{BBMS}.
Killing the generators $B$ in $\mathcal{P}$ gives a presentation of $G_B$
of the form $\<X\sqcup A\mid S'_1\cup S_2\cup S_4\>$ with $S'_1=\{ru_r(A)\colon r\in R\}$,
as in \cite[\S 1.4]{BBMS}.  By \cite[Theorem 1.2]{BBMS}, the normal closure
$N_A/(N_A\cap N_B)$ of $A$ in $G_B$ has a presentation on the generators $A$
with relators all the $F(X,A)$-conjugates of
$S_2\cup S_4\cup S_6\cup Z$, where $S_6=\{[ru_r,a]\colon r\in R,a\in A\}$
and where $Z$ is a finite set of words in $A^{\pm 1}$ derived  by a simple algorithm
from a finite set of
identity sequences that generate  $\pi_2(\mathcal{Q})$ as a $\Z Q$-module. 
(Note: this is where input datum (4) enters.)
  Now the relators $S_2$ and $S_4$ are already
relators of $\mathcal{P}$, while the relators $S_6$ can readily be derived from
the relators $S_1\cup S_2\cup S_3$ of $\mathcal{P}$.  It follows that the elements of $N_A$
represented by the words in $Z$ generate $N_A\cap N_B$ as a normal subgroup of $G$ (in other
words as a $\Z Q$-module), as required.     
\end{proof}

\begin{remark}\label{r:123map} The algorithm in the preceding proof does not just produce a finite presentation $\P\equiv \< T\mid 
\Sigma \>$ of the
fibre product $P<\G_1\times\G_2$, it also produces an explicit isomorphism $\Phi:|\P|\to P$ (induced by $(\phi_A\times\phi_B)\circ\Delta$
in the notation of the proof). It follows that if one has a preferred finite generating set $Y$ for $P$, then one can construct a finite
presentation for $P$ with generators $Y$. Indeed, a naive search will identify, for each generator $t\in T$, a word $u_t$ so that
$\Phi (t)=u_t(Y)$ in $P$. Then $P=\<T\sqcup Y\mid \Sigma,\, t^{-1}u_t(Y) \, (t\in T)\>$, and obvious Tietze moves remove the generators $T$.
\end{remark}

\begin{remark}\label{r:needPi2} There does not exist an algorithm that, on input a finite presentation of a group of type
${\rm{FP}}_\infty$ can output a finite set of module generators for $\pi_2$ of the presentation, so the last piece of 
input data in the above theorem cannot easily be dispensed with. In fact, Bridson and Wilton \cite{BWilt2} have proved
that it is essential: there exists a recursive sequence of maps $\phi_n:\Gamma_n\to Q_n$, with $\G_n$ and $Q_n$ given
by finite presentations, such that each $Q_n$ is of type ${\rm{FP}_\infty}$ and each kernel $\ker \phi_n$ is finitely generated,
but the first Betti number of the associated (finitely presentable) fibre product $P_n<\G_n\times\G_n$ cannot be determined algorithmically (whereas
it could be if one had a finite presentation in hand).
\end{remark}

\section{Subdirect products and VSP}\label{s:sdp}

Throughout this section we consider subdirect products of arbitrary
finitely presentable groups.  In later sections we restrict attention
to the case where the direct factors are limit groups.

Given a direct product $D:=G_1\times\cdots\times G_n$,
we shall consistently write $p_i$ and $p_{ij}$
for the projection
homomorphisms $p_i:D\to G_i$ and $p_{ij}: D\to G_i\times G_j$
($i,j=1,\dots,n$). 
We implicitly assume that $n\ge 2$.

We remind the reader that a subgroup $S<D$
is said to be  VSP ({\em{virtually surjective on pairs}}) if for all $i,j\in\{1,\dots,n\}, i\neq j$, the
projection $p_{ij}(S)<G_i\times G_j$ has finite index.

\begin{thm}[= Theorem \ref{t:pairs}]  
Let $S< G_1\times\dots\times G_n$ be a
subgroup of a direct product of finitely presented groups. If $S$ is VSP, then $S$ is finitely presentable.
\end{thm}

We will deduce this theorem from the Asymmetric 1-2-3 Theorem by
combining some well-known facts about virtually nilpotent groups
with the following proposition, which generalises similar results in
 \cite{BM} and \cite{BHMS1}.

\begin{prop}\label{p:nilp}
 Let $G_1,\dots,G_n$ be groups
and let $S< G_1\times\dots\times G_n$ be a subgroup.
If $S$ is VSP then 

\begin{enumerate}
\item there exist finite-index
subgroups $G_i^0 \subset G_i$ such that $\gamma_{n-1}(G_i^0)\subset S$.
\end{enumerate}
If, in addition, the groups $G_i$ are all finitely generated, then 
\begin{enumerate}
\item[(2)] $L_i:=S\cap G_i$ is finitely generated as a normal subgroup of $S$, 
\item[(3)]  $N_i:=S\cap \ker (p_i) $ is finitely generated, and 
\item[(4)] $S$ is itself finitely generated.
\end{enumerate}
\end{prop}

\begin{proof} 
The conditions imply that $p_i(S)$ is a finite index subgroup of $G_i$, and by passing to subgroups of finite index
we may assume without loss that $S$ is subdirect.

Let  
$$G_1^0=\{g\in G_1 \mid \forall j\neq 1 \, \exists (g,*,\dots,*,1,*\dots)\in
N_j\}=\bigcap_{j=2}^n \left(p_{1j}(S)\cap G_1\right)$$
and define $G_i^0$ similarly. As $p_{ij}(S)\subset G_i\times G_j$ is of
finite index, $G_i^0$ has finite index in $G_i$ for $i=1,\dots,n$.

For notational convenience we focus on $i=1$ and explain why
$\gamma_{n-1}(G_1^0)\subset S$. The key point to observe is that for all
$x_1,\dots,x_{n-1}\in G_1^0$ the commutator $([x_1,x_2,\dots,x_{n-1}],1,
\dots,1)$ can be expressed as the commutator of elements from the 
subgroups $N_j\subset
S$; explicitly it is
$$
[\, (x_1,1,*,\dots,*),\ (x_2,*,1,*,\dots,*),\dots,(x_{n-1},*,\dots,*,1)\,
].
$$
This proves the first assertion.

For (2), note that since $S$ is subdirect, $S\cap G_i$ is
normal in $G_i$ and the normal closure in $G_i$ of any
set $T\subset S\cap G_i$ is the same as its normal closure
in $S$. Since $G_i$ is finitely generated, 
$G_i/(S\cap G_i)$ is a finitely generated virtually
nilpotent group; hence it is finitely presented and
$S\cap G_i$ is the normal closure in $G_i$
(hence $S$) of a finite subset.

Towards proving (3), note that  the image of  $N_1=S\cap \ker (p_1)$
in $G_i$ under the projection $p_i$ has finite index for $2\le i\le n$, since
$p_{1i}(S)$ has finite index in $G_1\times G_i$ and $N_1$ is the 
kernel of the restriction to $S$ 
of $p_1=p_1\circ p_{1i}$.   
In particular $p_i(N_1)$ is finitely generated.

Note also that $L_i= S\cap G_i$ is the normal closure of
a finite subset of $p_i(N_1)$ by (2).

Now let $L:=L_2\times\cdots\times L_{n}$.
Then
$N_1/L$ is a subgroup of the finitely generated virtually nilpotent
group
$$\frac{G_2\times\cdots\times G_{n}}{L}
\cong\frac{G_2}{L_2}\times\cdots\times\frac{G_{n}}{L_{n}},$$
and hence is also finitely generated (and virtually nilpotent).

Putting all these facts together, we see that we can choose a finite subset $X$ of $N_1$ such that:
\begin{itemize}
\item $p_i(X)$ generates $p_i(N_1)$ for each $i=2,\dots,n$;
\item $X\cap L_i$ generates $L_i$ as a normal subgroup of $p_i(N_1)$,
for each $i=2,\dots,n$;
\item $\{xL:~x\in X\}$ generates $N_1/L$.
\end{itemize}
These three properties ensure that $X$ generates $N_1$, and the 
proof of (3) is complete.

We can express  $S$   as an  extension of $N_1$
by $G_1$ which are both finitely generated (using (3)), and (4) follows immediately.
\end{proof}

\begin{remarks}\label{ross}  

(1) A slight variation on the above proof of (3) shows that if
$G_1,\dots,G_n$ are finitely generated groups and $H<G_1\times\dots\times G_n$
is a subdirect product whose intersection with each of the
factors $G_i$ contains some term of the lower central
series of a subgroup of finite index in $G_i$, then $H$ is finitely
generated.

(2)
Finitely generated
virtually nilpotent groups are $\F_\infty$,
i.e. have classifying spaces with finitely many
cells in each dimension. Indeed this is true of virtually polycyclic
groups $P$, because such a group has a torsion-free subgroup of finite
index that is poly-$\Z$, and hence is the fundamental group of a
closed aspherical manifold. 
If $B$ has type $\F_\infty$ (e.g.
a finite group) and $A$ has type $\F_\infty$ (e.g. the fundamental
group of an aspherical manifold), then  any extension 
of $A$ by $B$ is also of type $\F_\infty$ (see \cite{ross} Theorem 7.1.10).
\end{remarks}

\subsection{Proof of Theorem \ref{t:pairs}} The hypothesis on $p_{ij}(S)$
implies that the image of $S$ in each factor $G_i$ is of finite index.
Replacing the $G_i$ and $S$ with subgroups of finite index
does not
alter their finiteness properties. Thus we may assume that
$S$ is a subdirect product. Let $L_i = G_i\cap S$ and note that $L_i$
is normal in both $S$ and $G_i$. Proposition \ref{p:nilp} tells us that
$Q_i:=G_i/L_i$ is virtually nilpotent; in particular it is of type $\F_3$
(see Remark \ref{ross}).

Assuming that $S$ is a subdirect product, we proceed by induction on $n$.
The base case, $n=2$, is trivial.

Let $q:G_1\times\cdots\times G_n\to G_1\times\cdots\times G_{n-1}$
be the projection with kernel $G_n$ and let  $T=q(S)$. By the inductive hypothesis,
$T$ is finitely presented.  We may regard $S$ as a subdirect product of
$T\times G_n$. Equivalently, writing 
 $N_n=T\cap S$ and noting that 
$$\frac{T}{N_n}\cong\frac{S}{N_n\times L_n}\cong\frac{G_n}{L_n}= Q_n,$$
we see that $S$ is  the fibre product associated to the
short exact sequences $1\to N_n \to T \to Q_n \to 1$ and
$1\to L_n\to G_n \to Q_n\to 1$. Thus, by the Asymmetric 1-2-3 Theorem,
our induction is  complete because according to Proposition \ref{p:nilp}(3), $N_n$ is
finitely generated.  \qed

\subsection{Separability}

It remains to prove the assertion in the last phrase of Theorem \ref{t:pairs}.

\begin{lemma}\label{l:nilp-sep} Let $D=G_1\times\dots\times G_n$ where the $G_i$
are finitely generated. If a subgroup $H<D$ is such that, for each $i$,
 $H\cap G_i$
contains a subgroup $N_i$ that is normal in $G_i$ with virtually nilpotent quotient,
then $H$ is separable in $D$.
\end{lemma}

\begin{proof} Let $N=N_1\times\dots\times N_n$. Then $D/N$ is virtually nilpotent,
hence subgroup separable. So given $g\in D\ssm H$ we can find 
a finite-index subgroup $K$ in $D/N$ that contains $H/N$ but not $gN$ (noting that
$N\subset H$). Then $KN$ has finite index in $D$ and separates $H$ from $g$.
\end{proof}

The following corollary completes the proof of Theorem \ref{t:pairs}.

\begin{corollary}\label{c:VSPsep}
If $H<D$ has the VSP property, then it is separable.
\end{corollary}
 
\begin{proof} We replace the subgroup $G^0_i<G_i$ of 
Proposition 3.2(1) with the intersection of all of its conjugates; it is then
normal in $G_i$, as is each term of its lower central series. Thus it suffices
to take $N_i = \gamma_{n-1}(G^0_i)$.
\end{proof}

 In our solution to the membership problem for finitely presented subgroups
of residually free groups, we shall need a further consequence of Lemma \ref{l:nilp-sep}.

\begin{corollary}\label{c:sdp-sep} If $G_1,\dots,G_n$ are limit groups, then every finitely presented full
subdirect product of $D=G_1\times\dots\times G_n$ is separable in $D$.
\end{corollary}

\begin{proof} 
We proved in Theorem 4.2 of \cite{BHMS1} that finitely presented full subdirect products of non-abelian limit groups have the VSP property,
so it only remains to deal with the abelian factors. By collecting these factors, we may
assume that precisely one factor ($G_n$ say) is abelian.
If $S<D$ is a finitely presented, full subdirect product, then its projection $q_n(S)$
in $G_1\times\cdots\times G_{n-1}$ is also finitely
presentable, full and subdirect, and hence satisfies VSP.

The proof of Proposition \ref{p:nilp}(1) applies to $S<D$, even if it is
not VSP: given $a_1,\dots,a_{n-2}\in G_i^0$,
that proof gives $s_1,\dots,s_{n-2}\in S$ such that
$p_i([s_1,\dots,s_{n-2}])=[a_1,\dots,a_{n-2}]$ in $G_i$,
and $p_j([s_1,\dots,s_{n-2}])=1$ in $G_j$ for $1\le j\le n-1$, $j\ne i$.
But it is also true that $p_n([s_1,\dots,s_{n-2}])=1$ in $G_n$,
since $G_n$ is abelian.  Hence $\gamma_{n-2}(G_i^0)\subset \gamma_{n-2}(S)$.

As in the proof of Corollary \ref{c:VSPsep}
above, we arrange that $G_i^0$ is normal and define $N_i=\gamma_{n-2}(G_i^0)$.
(When $i=n$ this gives $N_n=1$.)  Lemma \ref{l:nilp-sep} now completes the proof.
\end{proof}

\subsection{The effective version}
The following theorem will play a key part in our proof that the class of finitely-presentable 
residually-free groups is recursively enumerable.

\begin{theorem}\label{t:effFPsubdirect}
There exists a Turing machine that, given a finite
collection $G_1,\dots,G_n$ of finitely presentable groups (each given by
an explicit finite presentation) and a finite
subset $Y\subset G_1\times\cdots\times G_n$ (given as a set
of $n$-tuples of words in the generators of the $G_i$) such that
each projection $p_{ij}(Y)$ generates a finite-index subgroup
of $G_i\times G_j$ ($1\le i<j\le n$), will output a finite
presentation $\<Y\mid R\>$ for $S:=\<Y\>$. 
\end{theorem}

\begin{proof}
With the effective Asymmetric 1-2-3 Theorem (Theorem \ref{t:123eff})
in hand, we
follow the proof of Theorem \ref{t:pairs}.  As in Theorem \ref{t:pairs}
we first replace each $G_i$ by the finite-index subgroup $p_i(S)$ to get to a
situation where $S$ is subdirect.  
Here we use the Todd-Coxeter and Reidemeister-Schreier
processes to replace  the given presentations of the $G_i$
by presentations of the appropriate finite-index subgroups.
By using Tietze transformations we may 
take $p_i(Y)$ to be the generators
of this presentation.
Thus we express the revised 
$G_i$ as quotients of the free group on $Y$.

We argue by induction on $n$. The initial  
case $n=2$ is easily handled by the Todd-Coxeter
and Reidemeister-Schreier processes, since then
$S$ has finite index in the direct product.  So we may assume that $n\ge 3$.

By Theorem \ref{t:123eff} and Remark \ref{r:123map} it is sufficient to find finite presentations
for
\begin{enumerate}
\item $T=q(S)$, where $q$ is the natural projection from
$G_1\times\cdots\times G_n$ to $G_1\times\cdots\times G_{n-1}$, 
\item $G_n$, and
\item $Q=G_n/(G_n\cap S)$,
\end{enumerate}
together with
\begin{enumerate}
\item[(4)] explicit epimorphisms $T\to Q$ and $G_n\to Q$, and
\item[(5)] a finite set of generators for $\pi_2$ of the presentation
for $Q$, as a $\Z Q$-module.
\end{enumerate}

A finite presentation $G_n=\langle Y\mid R_n\rangle$ is part of the input.

\medskip
We may assume inductively that we have found a finite presentation for
$T$, with generators  $q(Y)$.
We write this presentation as $\<Y\mid r_1(Y),\dots, r_m(Y)\>$.
\medskip

To obtain a finite presentation for $Q$, we proceed as follows.
The image in $G_n$ of the words $r_j(Y)$ normally generate $G_n\cap S$.
Thus adding these words as relations to the existing presentation of $G_n$
gives a finite presentation of $Q$, together with the natural
quotient map $G_n\to Q$.

\medskip
The epimorphism $T\to Q$ is induced by the identity map 
on $Y$. 

\medskip
We would now be done if we could
 compute a finite set of $\pi_2$-generators for our
chosen finite presentation $\mathcal{P}$
of the virtually nilpotent group $Q$.
But it is more convenient to
proceed in a slightly different manner, modifying $\mathcal{P}$.

First, we search among finite-index normal subgroups $Q'$ of $Q$ for an
isomorphism $Q'\to P$, for some group $P$ given by a poly-$\Z$ presentation
$\mathcal{P}'$.
The latter presentation defines an explicit construction for a finite
$K(P,1)$-complex $X$, and in particular a finite set $B$ of generators of $\pi_2(X^{(2)})$ as a $\Z P$-module
(the attaching maps of the $3$-cells).

We next replace our initial presentation $\mathcal{P}$ for $Q$ by a
new presentation $\mathcal{Q}$ that contains $\P'$ as a sub-presentation.
Indeed, we know that such presentations exist, so we can find one, together
with an explicit isomorphism that extends the given isomorphism
$P\to Q'$, by a naive search procedure.

Let $K$ denote the presentation $2$-complex associated to the presentation
$\mathcal{Q}$, $\widehat{K}$ the regular cover of $K$ corresponding to
the normal subgroup $P=Q'$, and $Z$ the preimage of $X^{(2)}\subset K$
in $\widehat{K}$.  Then $Z$ consists of one copy of $X^{(2)}$ at each vertex of
$\widehat{K}$; these are
 indexed by the elements of the finite quotient group
$H=Q/Q'$.

We then have an exact homotopy sequence
$$\cdots \to \Z Q \otimes_{\Z Q'} \pi_2(X^{(2)}) \to \pi_2(\widehat{K}) \to \pi_2(\widehat{K},Z) \to 0$$
(since the map $P\to Q$ is injective by hypothesis), together with a finite
set $B$ of generators for $\pi_2(X^{(2)})$ as a $\Z Q'$-module (via our isomorphism
$Q'\to P$).
But $\pi_2(\widehat{K},Z)\cong H_2(\widehat{K}/Z)$, since the quotient complex
$\widehat{K}/Z$ is simply connected.  Hence $\pi_2(K)=\pi_2(\widehat{K})$
is generated as a $\Z Q$-module by $B$ together with any finite set $C$
that maps onto a generating set for the finitely generated abelian
group $H_2(\widehat{K}/Z)$.  
Such a set $C$ can be found by a naive search
over finite sets of identity sequences over  $\mathcal{Q}$.
\end{proof}

\begin{addendum}\label{a:noHalt} The Turing machine in Theorem \ref{t:effFPsubdirect} will not halt
on input of a subset $Y$ such that $p_{ij}(Y)$ generates a subgroup of infinite index in $G_i\times G_j$
$i\neq j$.
\end{addendum}

\begin{proof} The Todd-Coxeter process invoked in the preceding proof will not halt.
\end{proof}

\section{Novel Examples}\label{s:exs}

From \cite{BHMS1} (or \cite{BM} in the case of surface groups)
we know that a finitely presented full subdirect product $S$ of $n$ limit
groups $\G_i$  must virtually contain the term 
$\g_{n-1}$ of the lower central series of the product.  
So the quotient groups $\G_i/(S\cap\G_i)$ are 
virtually nilpotent of class at most  $n-2$.  
In particular for $n=3$ the quotients $\G_i/(S\cap\G_i)$
are virtually abelian.

A question left unresolved in \cite{BM} is whether a finitely presented subdirect product $S$
of $n$ free groups $\Phi_i$ can have $\Phi_i/(S\cap \Phi_i)$ nilpotent strictly of class 2 or more (necessarily
$n\geq 4$ for this to happen).  Theorem \ref{t:sec} below settles this question
and shows that the bounds on the nilpotency class given in 
\cite{BHMS1} and \cite{BM} are optimal.

\subsection{The groups $S(E,c)$}
The first part of our discussion applies to groups of a rather
general nature, but since our main interest lies with subgroups
of direct products of free groups, we fix the suggestive notation
$F$ for a finitely presentable group with a fixed generating set $\{a,b\}$.
(The restriction to the 2-generator case is just for notational convenience.)
Let $\Phi=F^\Z$
denote the unrestricted direct product of a countably infinite
collection of copies of $F$, thought of as the set of functions
$f:\Z\to F$ endowed with pointwise multiplication.

Let $\G=\<w,x,y,z\>$ be a free group of rank $4$, and define a
homomorphism $\phi:\G\to\Phi$ by $\phi(w)(n)=a$, $\phi(x)(n)=b$,
$\phi(y)(n)=a^n$, $\phi(z)(n)=b^n$ for all $n\in\Z$.

Given a finite subset $E\subset\Z$, we may regard the direct product
of $|E|$ copies of $F$ as the set $F^E$ of functions $E\to F$. We then obtain
a projection $p_E:\Phi\to F^E$ by restriction: $p_E(f)=f|_E:E\to F$.

Notice that when $E=\{n\}$ is a singleton  $p_E\circ\phi$ is surjective.
It will be convenient to write $\Phi_n$ for $F^{\{n\}}$, $p_n$ for the projection
$p_{\{n\}}:\Phi\to \Phi_n$, and $a_n,b_n$ for the
copy of $a,b$ respectively in $\Phi_n$.  The surjectivity of $p_n\circ\phi$
means that, for any finite subset $E\subset\Z$, the image of $p_E\circ\phi$ is a finitely
generated subdirect product of the groups $\Phi_n$ ($n\in E$).

This subdirect product is not in general finitely presented.

Now let $c$ be a positive integer.
We may choose a finite set $R=R(a,b)$ of normal generators for the $c$'th term
$\gamma_c(F)$ of the lower central series of $F$.  We then define $S(E,c)$ to
be the subgroup of $F^E$ that is generated by $(p_E\circ\phi)(\G)$
together with the sets $R(a_n,b_n)\subset \Phi_n$ for each $n\in E$.

As a concrete example we note that  $S(\{1,2,3,4\},3)$ is
the subgroup of $ \Phi_1\times  \Phi_2\times  \Phi_3\times  \Phi_4$ generated by the following
$12$ elements:  the four images of the generators of $\G$
$$(a_1,a_2,a_3,a_4) \thickspace,\thickspace  (b_1,b_2,b_3,b_4) \, $$
$$ (a_1, a_2^2,a_3^3,a_4^4) \thickspace,\thickspace (b_1,b_2^2,b_3^3,b_4^4)$$
together with the eight elements 
$$ ([[a_1,b_1],a_1],1,1,1) \thickspace,\thickspace ([[a_1,b_1],b_1],1,1,1) \thickspace,\thickspace
(1,[[a_2,b_2],a_2],1,1) \thickspace,\thickspace \ldots $$
$$ \ldots  \thickspace,\thickspace
(1,1,1,[[a_4,b_4],a_4]) \thickspace,\thickspace (1,1,1,[[a_4,b_4],b_4]) $$
which are normal generators for the subgroups $\gamma_3( \Phi_i)$ for $ 1\leq i\leq 4 $.

\begin{prop}\label{p:sec}
The groups $S(E,c)$ have the following properties.
\begin{enumerate}
 \item\label{p1} $S(E,c)$ contains $\gamma_c(F^E)$.
 \item\label{p2} $S(E,c)$ is finitely presentable.
 \item\label{p3} If $E'=E+t=\{e+t;~e\in E\}$ is a translate of $E$ in $\Z$, then
$$S(-E,c)\cong S(E,c)\cong S(E',c).$$
 \item\label{p4} If $E\subset E'$, then the projection $F^{E'}\to F^E$ induces an
epimorphism $S(E',c)\to S(E,c)$.
\end{enumerate}
\end{prop}

\begin{pf}

(\ref{p1}) Since
$R(a_n,b_n)\subset S(E,c)\cap \Phi_n$ by construction,
and since $p_n\circ\phi$ is surjective for all $n\in E$, it follows
that $S(E,c)\supset\gamma_c(\Phi_n)$ for each $n\in E$, and hence
$S(E,c)\supset\gamma_c(F^E)$.

(\ref{p2})
Clearly $S(E,c)$ is finitely generated.
  For any $2$-element subset $T=\{m,n\}$
of $E$, the image of the projection of $S(E,c)$ to $F^T=\Phi_m\times \Phi_n$ is precisely
$S(T,c)$.  Since $S(T,c)$ contains the elements $p_T(\phi(w))=(a_m,a_n)$,
$p_T(\phi(x))=(b_m,b_n)$, $p_T(\phi(yw^{-m}))=(1,a_n^{n-m})$ and
$p_T(\phi(zx^{-m}))=(1,b_n^{n-m})$, together with $\gamma_c(\Phi_m\times \Phi_n)$,
we see that the quotient of each of the direct factors $\Phi_m\cong F\cong \Phi_n$
by its intersection with $S(T,c)$ is a nilpotent group of class at most $c$, generated
by two elements of finite order, and hence is finite.  Thus $S(T,c)$ has finite
index in $F^T$.  In other words, the projection of the subdirect product
$S(E,c)<F^E$ to each product of two factors $F^T$ has finite index.  Hence
by Theorem \ref{t:pairs}, $S(E,c)$ is finitely presentable.

(\ref{p3})
It is clear that $S(-E,c)\cong S(E,c)$ via the isomorphism $F^E\to F^{-E}$
defined by $a_n\mapsto a_{-n}$, $b_n\mapsto b_{-n}$.

To show that $S(E,c)\cong S(E',c)$,
it is clearly enough to consider the case $t=1$.  The isomorphism 
$\theta:F^E\to F^{E'}$ defined by $a_n\mapsto a_{n+1}$, $b_n\mapsto b_{n+1}$
is induced by the shift automorphism $\overline\theta:\Phi\to\Phi$ defined by
$\overline\theta(f)(k):=f(k-1)$, in the sense that
$p_{E'}\circ\overline\theta=\theta\circ p_E$.

Similarly, $\overline\theta$ commutes with the automorphism $\widehat\theta$
of $\G$ defined by $w\mapsto w$, $x\mapsto x$, $y\mapsto yw^{-1}$, $z\mapsto zx^{-1}$,
in the sense that $\overline\theta\circ\phi=\phi\circ\widehat\theta$.
It follows immediately from the definitions that $\theta$ maps $S(E,c)$ onto $S(E',c)$.

(\ref{p4}) This is immediate from the definitions.
\end{pf}

We can now state and prove the main result of this section.
We thank Mike Vaughan-Lee for several helpful suggestions
concerning this proof.

Henceforth we assume that $F=\<a,b\>$ is a free group of rank 2.
(As above, this is purely for notatonal convenience; analogues of our examples
can be constructed in the same way using non-abelian free groups of arbitrary rank.) 

\begin{theorem}[= Theorem \ref{t:exs}]\label{t:sec}
For any positive integer $c$, and any finite subset $E\subset\Z$ of cardinality
at least $c+1$, the group
$S(E,c)$ is a finitely presentable subdirect product of the non-abelian free groups
$\Phi_n$ ($n\in E$) and $S(E,c)\cap \Phi_n=\gamma_c(\Phi_n)$ for each $n\in E$.
\end{theorem}

\begin{pf}
By construction, $S(E,c)$ is a subdirect product of the $\Phi_n$ for $n\in E$,
and by Proposition \ref{p:sec}(\ref{p2}) it is finitely presentable.
By Proposition \ref{p:sec}(\ref{p1}) we have $$S(E,c)\cap \Phi_n\supset\gamma_c(\Phi_n)$$
for each $n\in E$, so it only remains to prove the reverse inclusion.

Let $A=\Q[[\a,\b]]$ be the algebra of power series in two non-commuting
variables $\a,\b$ with rational coefficients, and for each $n$ let $\eta_n:\Phi_n\to U(A)$ be the
Magnus embedding of $\Phi_n$ into the group of units $U(A)$ of $A$, defined by
$\eta_n(a_n)=1+\a$, $\eta_n(b_n)=1+\b$.  By Magnus' Theorem \cite{Mag} (or \cite[Chapter 5]{MKS}),
$\eta_n^{-1}(1+J^c)=\gamma_c(\Phi_n)$.  Here $J$ is the ideal generated by the 
elements with 0 constant term and $J^c$ is its $c$-th power.

Now define $\eta:\Gamma\to U(\Q[t]\otimes_\Q A)$ by $\eta(w)=1+\a$, $\eta(x)=1+\b$,
$\eta(y)=(1+\a)^t$, $\eta(z)=(1+\b)^t$, where for example $(1+\a)^t$ means the power series
$$(1+\a)^t=\sum_{k=0}^\infty \left(\begin{array}{c} t\\ k\end{array}\right) \a^k
=\sum_{k=0}^\infty\frac{t(t-1)\cdots (t-k+1)}{k!}~\a^k.$$

Note that $\eta_n\circ\phi_n=\psi_n\circ\eta$,
where $\psi_n:\Q[t]\otimes_\Q A\to A$ is defined by $f(t)\otimes a\mapsto f(n)a$
and where $\phi_n = p_n\circ \phi$.

Note also that, for any $g\in\Gamma$, $\eta(g)$ has
the form $$\eta(g)=\sum_{W\in\Omega} \pi_W(t)\cdot W(\a,\b),$$ where
$\Omega$ is the free monoid on $\{\a,\b\}$ and $\pi_W(t)\in\Q[t]$
has degree at most equal to the length of $W$.
Hence, for each $n\in\Z$, we have
$$\eta_n(\phi_n(g))=\psi_n(\eta(g))=\sum_{W\in\Omega} p_W(n)\cdot W(\a,\b).$$

Suppose now that $E\subset\Z$ is a finite set of integers of cardinality
at least $c+1$, and that 
$g\in\G$ such that $p_E(\phi(g))\in S(E,c)\cap \Phi_n$ for some $n\in E$.
Then, for each $m\in E\smallsetminus\{n\}$, we have 
$$\psi_m(\eta(g))=\eta_m(\phi_m(g))=\eta_m(1)=1.$$

It follows that, in the expression $\eta(g)=\sum_{W\in\Omega} \pi_W(t)\cdot W(\a,\b)$
for $\eta(g)$, the elements of $E\smallsetminus\{n\}$, of which there are at least $c$, are roots of
all the polynomials $\pi_W(t)$.  In particular, for words $W$ of length less than
$c$, the polynomials $\pi_W$ are identically zero.  Hence $\psi_m(\eta(g))\in 1+J^c$
for all $m\in\Z$, in particular for $m=n$.
Hence $\phi_n(g)\in\eta_n^{-1}(1+J^c)=\gamma_c(\Phi_n)$.

Thus $$S(E,c)\cap \Phi_n\subset\gamma_c(\Phi_n),$$ completing the proof
that $$S(E,c)\cap \Phi_n=\gamma_c(\Phi_n).$$
\end{pf}

\subsection{Sample calculations}

We use the explicit form of the 
map $\eta:\Gamma\to U(\Q[t]\otimes_\Q A)$ from the proof of Theorem \ref{t:sec} 
to make some calculations that illuminate the preceding proof. Recall that $\G=\<w,x,y,z\>$
is free of rank 4. 

\begin{remark}\label{rk:comms}
Suppose that $U,V\in\G$, $k,\ell\ge 1$ and $\a\in J^k$, $\b\in J^\ell$ are such that
$\eta(U)=1+\a~\mathrm{mod}~J^{k+1}$, $\eta(V)=1+\b~\mathrm{mod}~J^{\ell+1}$.
Then $\eta(UV)-\eta(VU) = \a\b-\b\a~\mathrm{mod}~J^{k+\ell+1}$,
while $\eta(U^{-1}V^{-1})=1~\mathrm{mod}~J$, so
$$\eta([U,V])-1=\eta(U^{-1}V^{-1})(\eta(UV)-\eta(VU))=\a\b-\b\a~\mathrm{mod}~J^{k+\ell+1} \ .$$
\end{remark}

\begin{example}
For each integer $k$, we calculate that
$$\eta(zx^{-k}) = 1 + (t-k)\b ~\mathrm{mod}~J^2.$$
Also
$$\eta(y) = 1 + t\a ~\mathrm{mod}~J^2.$$
Repeatedly applying Remark \ref{rk:comms}, we see that
$$\eta([y,zx^{-1},zx^{-2},\dots,zx^{-m}])=1 + t(t-1)\cdots (t-m)V_m(\a,\b)~\mathrm{mod}~J^{m+2},$$
where $$V_m:=\sum_{k=1}^m \ch{m}{k}\b^k\a\b^{m-k}$$
is a non-trivial $\Z$-linear combination of homogeneous monomials of degree $m+1$.

Notice that the coefficient of $V_m(\a,\b)$ is a polynomial of degree $m+1$ in $t$
with roots $0,1,\ldots,m$.  In particular this gives an example of an element in
$S(\{0,\ldots,m+1\}, m+2) \cap \gamma_{m+1}( \Phi_{m+1})$ 
which is not in $\gamma_{m+2}(\Phi_{m+1})$.
\end{example}

\begin{example}
As another application of Remark \ref{rk:comms}, we see inductively
that, for any basic commutator $C$ of weight $c$ in the generators of $\G$,
$$\eta(C)\in \Z[t][[\a,\b]] + J^{c+1},$$
and hence
$$\eta(\gamma_c(\G))\subset \Z[t][[\a,\b]] + J^{c+1}.$$

On the other hand, if we put $U=[w,z][x,y]\in\gamma_2(\G)$, then
$$\eta(U)=1+\ch{t}{2}(\a\b^2+\b^2\a+\b\a^2+\a^2\b-2\a\b\a-2\b\a\b)~\mathrm{mod}~J^4.$$
Thus $\phi(U)$ is an element of $\gamma_3(S(E,c))$ for any $E,c$.  On the other hand,
since $\ch{t}{2}\notin\Z[t]$, $\eta(U)\notin\eta(\gamma_3(\G))$, so for sufficiently
large $E,c$ the element $\phi(U)\in\gamma_3(S(E,c))$ does not belong to
$\phi(\gamma_3(\G))$.
\end{example}

\section{Characterizations}\label{s:char}

In this section we discuss the structure of finitely presentable
residually free groups, and prove some results concerning their
classification.

\subsection{Subdirect products and homological finiteness properties}

We remind the reader of the shorthand we introduced in order to state Theorem \ref{t:main}
concisely: an embedding $S\hookrightarrow
\G_0\times\dots\times\G_n$ of a residually free group $S$
as a full subdirect product of  limit groups is said to be {\em{neat}}
if $\G_0$ is
abelian, $S\cap\Gamma_0$ is of finite index in $\G_0$, and
$\G_i$  is non-abelian for $i=1,\dots,n$.

\begin{thm}[=Theorem \ref{t:main}]\label{t:5.1} Let $S$ be a  finitely generated
residually free group.  Then the  following conditions are equivalent:
\begin{enumerate}
\item $S$ is finitely presentable;
\item $S$ is of type $\mathrm{FP}_2(\Q)$;
\item ${\rm{dim}\,}H_2(S_0;\Q)<\infty$ for all subgroups $S_0\subset S$ of
finite index;
\item  there exists a neat embedding of $S$ as a full subdirect of finitely many limit groups
so that the image is VSP;
\item  the image of every neat embedding of $S$ as a full subdirect of finitely many limit groups
is VSP.
\end{enumerate} 
\end{thm}

\begin{proof}

The implications (1) implies (2) implies (3) are clear.
Theorem \ref{t:pairs} shows that (4) implies (1).

In order to establish the remaining implications, we first argue that 
every finitely generated residually free group
has a neat embedding. Now   \cite[Corollary 19]{BMR} tells us that
$S$ embeds into the direct product of a finite collection of limit
groups.  Since finitely generated subgroups of limit groups are limit groups,
we may assume that $S$ is a subdirect product of finitely many limit groups.
Moreover, by projecting away from any factor with which $S$ has trivial
intersection, we may assume that $S$ is a full subdirect product of limit
groups, say $S<\G_0\times\cdots\times\G_n$.  Moreover, if two or more of the factors
$\G_i$ are abelian, we may regard their direct product as a single direct
factor, so we may assume that $\G_0$ is abelian (possibly trivial), and that
$\G_i$ is non-abelian for $i>0$.   Finally, the intersection $S\cap\G_0$
has finite index in some direct summand of $\G_0$, and by projecting away from
a complement of such a direct summand, we may assume that $S\cap\G_0$
has finite index in $\G_0$. Thus we obtain a neat embedding of $S$.
With this existence result in hand, it is clear that (5) implies (4). To complete
the proof we shall argue that  (3) implies (5).

Given a neat embedding $S\hookrightarrow \Lambda_0\times\dots\times\Lambda_m$,
the image of the projection of $S$ to $\Lambda_0\times\Lambda_i$
has finite index for any $i>0$, and  the quotient $\overline S$
of $S$ by  $Z(S)=S\cap\Lambda_0$
is a  full  subdirect product of the non-abelian limit
groups $\Lambda_1,\dots,\Lambda_n$. Moreover, since $S\cap\Lambda_0$ is finitely
generated, (3) implies that $H_2(\overline S_0;\Q)$ is finite dimensional
for all subgroups $\overline S_0 < \overline S$ of finite index in $\overline S$.
It then follows from
Theorem 4.2 of  \cite{BHMS1} that the image of the projection of $S$ to
$\Lambda_i\times\Lambda_j$ has finite index for any $i,j$ with $0<i<j\le m$.
Thus (3) implies (5).  
\end{proof}

It follows easily from Theorem \ref{t:5.1} that any subdirect product
of limit groups   that contains a finitely presentable full subdirect product  
is again finitely presentable. More generally we prove:

\begin{thm}[= Theorem F]\label{t:5.2} Let $k \ge 2$ be an integer,
let $S\subset D:=\G_1\times\dots\times\G_n$ be a full subdirect
product of limit groups, and let $T\subset D$ be a subgroup that contains
$S$. If $S$ is of type ${\rm{FP}}_k(\Q)$ then so is $T$.
\end{thm}

\begin{proof}
 We have $S<T<D=\G_1\times\dots\times\G_n$ where the $\G_i$
are limit groups and $S$ is a full subdirect product of type ${\mathrm{FP}}_k(\Q)$
with $k\ge 2$.

In particular, $S$ is of type $\mathrm{FP}_2(\Q)$, so by \cite[Theorem 4.2]{BHMS1}
the quotient group $D/L$ is virtually nilpotent, where $L=(S\cap\G_1)\times\dots\times (S\cap\G_n)$.

By \cite[Corollary 8.2]{BHMS1}, applied to $T/L$, there is a finite index subgroup $S_0<S$, and
a subnormal chain $S_0\triangleleft S_1\triangleleft \dots \triangleleft S_\ell=T$
such that each quotient $S_{i+1}/S_i$ is either finite or infinite cyclic.

Since $S$ is of type $\mathrm{FP}_k(\Q)$, so is $S_0$, and by the obvious induction 
so are $S_1,\dots,S_\ell=T$.
\end{proof}

Note that the condition $k\ge 2$ in Theorem \ref{t:5.2} is essential.  For example, if
$G=\< x,y | r_1,r_2,\dots\>$ is a $2$-generator group that is not finitely presentable,
then the subgroup $T$ of $F(x,y)\times F(x,y)$ generated by $\{(x,x),(y,y),(1,r_1),(1,r_2),\dots\}$
is a full subdirect product that is not finitely generated, while the finitely generated
subgroup $S$ of $T$ generated by $\{(x,x),(y,y),(1,r_1)\}$ is also a full subdirect product
(provided $r_1\ne 1$ in $F(x,y)$).  This is another example of the notable divergence in behaviour
between finitely presentable residually free groups and more general finitely generated residually free groups.

\subsection{The three factor case}

Theorem \ref{t:main} tells us which full subdirect products of non-abelian limit
groups are finitely presentable.  In the case of two factors,
the criterion is particularly simple: the subgroup must have finite
index in the direct product.  Our next result, which extends Theorem \ref{t:pairs}
 of \cite{BM}, shows that the criterion also
takes a particularly simple form in the case of a full subdirect product of three
non-abelian limit groups.  Our results in Section \ref{s:exs} show that the
situation is noticeably more subtle for subdirect products of four or more factors.

\begin{theorem}
Let $\G_1,\G_2,\G_3$ be non-abelian limit groups, and let
$S<\G_1\times\G_2\times\G_3$ be a full subdirect
product.  Then $S$ is finitely presentable if and
only if there are subgroups $\Lambda_i<\G_i$ of
finite index, an abelian group $Q$, and epimorphisms
$\phi_i:\Lambda_i\to Q$, such that
$$S\cap (\Lambda_1\times\Lambda_2\times\Lambda_3) =\ker (\phi),$$
where $$\phi:\Lambda_1\times\Lambda_2\times\Lambda_3\to Q,~~\phi(\lambda_1,\lambda_2,\lambda_3):=\phi_1(\lambda_1)+\phi_2(\lambda_2)+\phi_3(\lambda_3).$$
\end{theorem}

\begin{pf} 
First we argue that the
criterion in the statement is sufficient.
Each $\phi_i$ is an epimorphism, so given $\lambda_1\in\Lambda_1$
and $\lambda_2\in\Lambda_2$, there exists $\lambda_3\in\Lambda_3$ such that
$\phi_3(\lambda_3)=-\phi_1(\lambda_1)-\phi_2(\lambda_2)$.  Thus
$(\lambda_1,\lambda_2,\lambda_3)\in\ker (\phi)$, and the projection
$p_{12}:\G_1\times\G_2\times\G_3\to\G_1\times\G_2$ maps $\ker (\phi)$
onto the finite-index subgroup $\Lambda_1\times\Lambda_2$ of $\G_1\times\G_2$.
Similar arguments apply to the projections $p_{13}$ and $p_{23}$, so
the finite-index subgroup $\ker (\phi)$ of $S$ is finitely presentable,
by Theorem \ref{t:pairs}, and hence $S$ is also finitely presentable.

\medskip
Conversely, suppose that $S$ is finitely presentable.
By \cite[Theorem 4.2]{BHMS1}
the image of each of the projections $p_{ij}:S\to\G_i\times\G_j$
($1\le i<j\le 3$) has finite index.  The images of $p_{12}$
and $p_{13}$ intersect in a finite-index subgroup
$K_1<\G_1$. For each $a\in K_1$ there are elements $(a,1,x_a),(a,y_a,1)\in S$.
So given $a,b\in K_1$, we have $([a,b],1,1)=[(a,1,x_a),(b,y_b,1)]\in [S,S]$.
Thus $[K_1,K_1]<([S,S]\cap\G_1)$.  Similarly there are finite-index
subgroups $K_2<\G_2$ and $K_3<\G_3$ such that $[K_i,K_i]<([S,S]\cap\G_i)$
for $i=2,3$.   Let $A$ denote the abelian group
$$A=\frac{K_1\times K_2\times K_3}{S\cap (K_1\times K_2\times K_3)},$$
let $\phi:K_1\times K_2\times K_3\to A$ be the canonical epimorphism, and
let $\phi_i$ be the restriction of $\phi$ to $K_i$ for $i=1,2,3$.
Since $p_{23}(S)$ has finite index in $\G_2\times\G_3$, the same is true
of $p_{23}(S\cap (K_1\times K_2\times K_3))$ in $K_2\times K_3$.
Now let $\alpha=(x,y,z)\cdot (S\cap (K_1\times K_2\times K_3))\in A$.
For some positive integer $N$ we have
$(y^N,z^N)\in p_{23}(S\cap (K_1\times K_2\times K_3))$, so
$(w,y^N,z^N)\in S$ for some $w\in K_1$.  But then 
$\alpha^N=\phi_1(x^Nw^{-1})$, so $\phi_1(K_1)$ has finite index
in $A$.  Similarly, $\phi_2(K_2)$ and $\phi_3(K_3)$ have finite
index in $A$.  Let $Q$ be the finite-index subgroup
$\phi_1(K_1)\cap\phi_2(K_2)\cap\phi_3(K_3)$ of $A$, and
define $\Lambda_i=\phi_i^{-1}(Q)$ for $i=1,2,3$.  Then
$\Lambda_i$ has finite index in $\G_i$,
$S\cap (\Lambda_1\times\Lambda_2\times\Lambda_3)$
is the kernel of the restriction
$\phi:\Lambda_1\times\Lambda_2\times\Lambda_3\to Q$, and
each $\phi_i:\Lambda_i\to Q$ is an epimorphism.
\end{pf}

\subsection{Classification up to commensurability}

We construct a collection of examples of finitely presentable residually
free groups which is complete up to commensurability.  Our proof makes use of
the following observation on subgroups of finitely generated nilpotent groups.

\begin{lemma}\label{l:nilp_subgroup}
Let $G$ be a finitely generated nilpotent group and let $H$ be a subgroup of $G$.
If the image of $H$ in the abelianisation of $G$ has finite index, then $H$
has finite index in $G$. 
\end{lemma}

\begin{proof}
We argue by induction on the nilpotency class $c$ of $G$. If $c=1$ then $G$
is abelian and there is nothing to prove. By induction, we may assume that the 
image of $H$ in  $G/\gamma_c G$ has finite index. Thus  it suffices to show that $H\cap\g_c{G}$ has finite index
in $\g_c{G}$.  Now $\g_c{G}$ is an abelian group generated by finitely
many commutators $w=[x_1,\dots,x_c]$ of weight $c$, so it in fact suffices
to show that each has a power that lies in $\g_c{G}\cap H$.  If $|G:H\gamma_2 G|=n$
then we can write $x_i^n=h_iy_i$ for some $h_i\in H$ and some $y_i\in\gamma_2 G$.
An elementary commutator calculation using the fact that $\g_c{G}$ is central in $G$
yields the equality $$w^{n^c}=[x_1^n,\dots,x_c^n]=[h_1,\dots,h_c]\in H,$$
which completes the proof.
\end{proof}

\begin{definition}
Let $\mathcal{G}=\{\G_1,\dots,\G_n\}$ be a finite collection of $2$ or more
limit groups, let $c\ge 2$ be an integer, and let
$\underline{g}=\{(g_{k,1},\dots,g_{k,n}),~1\le k\le m\}$ be a finite
subset of $\G:=\G_1\times\cdots\times\G_n$.

Define
$T=T(\mathcal{G},\underline{g},c)$ to be the subgroup of 
$\G$ generated by $\underline{g}$ together with the
$c$'th term $\gamma_c(\G)$ of the lower
central series of $\G$.
\end{definition}

\begin{theorem}\label{t:cchar}
Let $T(\mathcal{G},\underline{g},c)$ be defined as above.
\begin{enumerate}
\item If, for all $1\le i<j\le n$, the images in $H_1\G_i\times H_1\G_j$ of the
ordered pairs $(g_{k,i},g_{k,j})$ generate a subgroup of finite index, then the
residually free group $T(\mathcal{G},\underline{g},c)$ is finitely presentable.
\item Every finitely presentable residually free group is either a limit
group or else is commensurable with
one of the groups $T(\mathcal{G},\underline{g},c)$.
\end{enumerate}
\end{theorem}

\begin{pf} 
To see that $T=T(\mathcal{G},\underline{g},c)$ is finitely presentable,
it is sufficient in the light of Theorem \ref{t:pairs}
to know that the projection of $T$ to
$\G_i\times\G_j$ is virtually surjective for each $i<j$, and this
follows from Lemma \ref{l:nilp_subgroup}

Conversely, suppose that $S$ is a finitely presentable residually free group.
If $S$ is not itself a limit group, then Theorem \ref{t:main}
tells us that $S$ may be expressed as a full subdirect product of limit
groups $\Delta_1,\dots,\Delta_n$ such that the projection of $S$ to $\Delta_i\times\Delta_j$
is virtually surjective for each $i<j$.  By Proposition \ref{p:nilp}(1),
each $\Delta_i$ contains a finite-index subgroup $\G_i$ such that
$\gamma_{n-1}(\G_i)\subset S$.  Set $\mathcal{G}=\{\G_1,\dots,\G_n\}$,
and $c=n-1$.  We choose any finite set
$\underline g= \{(g_{1,1},\dots,g_{1,n}),\dots,(g_{m,1},\dots,g_{m,n})\}$
in the direct product $D:=\G_1\times\cdots\times\G_n$ whose image in
$D/\gamma_{n-1}(D)$
generates 
$(S\cap D)/\gamma_{n-1}(D)$.
Then $T=T(\mathcal{G},\underline{g},n-1)=S\cap D$
is a finite-index subgroup of $S$.
\end{pf}

\section{The Canonical Embedding Theorem}\label{s:embed}
 
The purpose of this section is to prove Theorem \ref{t:ee(S)}: we shall
describe an effective construction for $\ee{S}$, hence 
$\eer{S}$, then establish the universal property of the
latter. We shall see that the direct factors of $\ee{S}$ are  the maximal limit group quotients of $S$:
the maximal free abelian quotient  $\go{S}$
is one of these, and the remaining (non-abelian)
quotients form $\eer{S}$.  At the end of the section we shall discuss how $\ee{S}$
is related to the Makanin-Razborov diagram for $S$.

Our first goal is to prove Theorem \ref{t:ee(S)}(1).
\begin{thm}\label{t:A1}
There is an algorithm that, given
a finite presentation of a residually free group $S$, will construct
an embedding  $$S\hookrightarrow \ee{S} = \G_{\rm ab} \times \eer{S}$$  
where $\G_{\rm ab}=\go{S}$ and 
$\eer{S} = \G_1\times\dots\times\G_n$ with each $\G_i\ (i\ge 1)$
a non-abelian limit group.
The intersection of $S$ with the 
kernel of the projection $\rho:\ee{S}\to\eer{S}$ 
is the centre $Z(S)$ of $S$.
\end{thm}

In outline, our proof of this theorem proceeds as follows. First we define a
finite set of data --- a {\em{maximal centralizer system}} --- which encodes a canonical
system of subgroups  in $S$. Then,
in Lemma \ref{mcs-exists},
we prove that every finitely presented residually free group possesses such a system; the proof, which is
not effective, relies on Proposition \ref{p:nilp} and results from \cite{BHMS1}.
In Lemma \ref{mcs-suffices} we establish the existence of a simple algorithm that, given
a maximal centralizer system, will construct $S\hookrightarrow \ee{S}$. Finally, in Subsection \ref{ss:A12},
we describe an algorithm that, given a finite presentation of a residually free group, will construct a maximal centralizer
system for that group (termination of the algorithm is guaranteed by Lemma \ref{mcs-exists}).

\smallskip

The description of $Z(S)$ given in Theorem \ref{t:A1} is covered by the following lemma.

\begin{lemma}\label{l:Z} Let $S$ be a residually free group
and let $Z(S)$ be its centre.
\begin{enumerate}
\item
The restriction of $S\to \go{S}$ to $Z(S)$ is injective.
\item If $\G$ is a non-abelian limit group and $\psi:S\to\G$
has non-abelian image, then $\psi(Z(S))=\{1\}$.
\end{enumerate}
\end{lemma}

\begin{proof} Let $\gamma\in Z(S)$. Since $S$ is residually free,
there is an epimorphism $\psi$ from $S$ to a free group such that
$\psi(\gamma)\neq 1$. But the only free group with a non-trivial centre is $\Z$,
so $\psi([S,S])=1$ and hence $\gamma\not\in [S,S]$. This observation, together
with the fact that residually free groups are torsion-free, proves (1).

Item (2) follows easily from the fact that limit groups are
commutative-transitive.
\end{proof}

\subsection{Centralizer systems}\label{s:chuck}

Before pursuing the strategy of proof outlined above, we present an auxiliary 
result that motivates the definition of a maximal centralizer system.
Recall that a set of subgroups of a group $H$ is said to be {\em characteristic}
if any automorphism of $H$ permutes the subgroups in the set.

\begin{prop} \label{charsubgp}
Let 
$D=\G_1\times\dots\times \G_n$ be a direct product of
non-abelian limit groups, let $S\subset D$ be a full subdirect product,
let $L_i=S\cap\G_i$
and let
$$M_i = S\cap(\G_1\times\cdots\times\G_{i-1}\times 1\times
\G_{i+1}\times \cdots\times\G_n).$$ 
The sets of subgroups $\{L_1,\dots,L_n\}$
and  $\{M_1,\ldots, M_n\}$ are characteristic in $S$. 
\end{prop}

\begin{proof} 
If $\G$ is a non-abelian limit group,  and
if $\g_1$ and $\g_2$ are two non-commuting elements of $\G$,
then the centralizer $C_\G(\g_1,\g_2)$ of the pair is trivial,
by commutative-transitivity.

The collection
of centralizers of non-commuting pairs of elements of $S$
has a finite set of maximal elements, namely the centralizers
of pairs $x_i$ and  $y_i$ which are non-commuting pairs in $L_i$.
These maximal elements are exactly the $M_i$, which therefore form a 
characteristic set.
Moreover the $L_i$ are the centralizers of the $M_i$ and hence the 
set of these is  also characteristic (cf. \cite{BM1}).\end{proof}

\begin{remark} Applying the proposition with $S=D$
one sees that if $D=\G_1\times\dots\times \G_n$  is the direct product of
non-abelian limit groups,
then the set of subgroups $\G_i$ is characteristic. 
In particular, the decomposition of $D$
as a direct product of limit groups is unique.

The example $D=\Z\times F_2$ shows  that this uniqueness fails  if abelian factors are allowed.
\end{remark}

\begin{definition}\label{d:MCS} Let $S$ be a finitely presented, non-abelian residually
free group.
A finite list $(Y_i;Z_i)=(Y_1,\ldots,Y_n ; Z_1,\ldots,Z_n)$ of finite subsets of $S$
will be called a {\em maximal centralizer structure (MCS)} for $S$ 
if it has the following properties.
\begin{enumerate}
\item[MCS(1)] Each $Y_i$ contains at least two elements $x_i$ and $y_i$ 
which do not commute.
\item[MCS(2)] Each $Z_i$ contains all of the $Y_j$ with $j\neq i$.
\item[MCS(3)] For each $i$, the elements of $Z_i$ commute with the elements of $Y_i$.
(Hence the elements in $Y_i $ commute with those in $Y_j$ for all $i\neq j$.)
\item[MCS(4)] Each $Z_i$ generates a normal subgroup of $S$.
\item[MCS(5)] For each $i$, the quotient group $S/\sg{Z_i}$  admits a splitting (as an
amalgamated free product or HNN extension)
either over the trivial subgroup or over a non-normal, infinite cyclic
subgroup.
\item[MCS(6)] There is a subgroup $S_0$ of finite index in $S$ such that each 
$Y_i\subset S_0$ and $S_0/\nc{Y_1,\ldots,Y_n}$ is nilpotent of class at most
$n-2$.
\end{enumerate}

For   the  case $n=1$  we require that 
$\<Z_1\>=Z(S)$  and that $Y_1$ be the given generating set for $S$.
\end{definition}

\begin{remark} One of the basic properties of non-abelian limit groups is 
that they split as in MCS(5). Conversely, 
we shall see in
Lemma \ref{mcs-suffices}  that, in the presence of the other conditions,
MCS(5) implies the following condition:

\begin{enumerate}
\item[\ MCS$(5)'$] For each $i$, the quotient $S/\sg{Z_i}$ is a non-abelian
limit group. 
\end{enumerate}
\end{remark}

\begin{lemma}\label{mcs-exists}
Every finitely presented non-abelian residually free group possesses a maximal centralizer structure.
\end{lemma}

\begin{pf}
Let $S$ be a finitely presented non-abelian residually free group, and 
define $H=S/Z(S)$. We shall first construct an MCS for $H$.

As in the proof of Theorem \ref{t:main},   $H$
can be embedded as a full subdirect product in some
$D=\G_1\times\dots\times\G_n$ where the $\G_i$ are non-abelian limit groups.
Let $p_i:D\to\G_i$ denote the projection.

If $n=1$, then $H$ itself is a non-abelian limit group.  In this case, we
follow the directions in the definition of MCS: $Y_1$ is the given set of generators for
$H$, $Z_1=\{1\}$, and $H_0=H$.  Then MCS(1-4) and MCS(6) are trivially satisfied,
as is  MCS$(5)'$, hence MCS$(5)$.

From now on we assume that $n>1$.
Then $H<D$ is VSP and  $\G_i/(H\cap\G_i)$ is virtually nilpotent, by \cite[Theorem 4.2]{BHMS1}, 
so $(H\cap\G_i)$ is finitely
generated as a normal subgroup of $\G_i$.  
Choose a finite set $Y_i$ of normal generators for $H\cap \G_i$
containing at least two elements that do not commute.

Let $M_i$ denote the centralizer of $Y_i$ in $H$ (this is consistent with the notation
in Proposition \ref{charsubgp}). Note that $M_i=H\cap \ker (p_i)$, which by Proposition \ref{p:nilp}(3)
is a finitely generated subgroup
of $H$. Note that $\G_i\cong H/M_i$.  
Choose $Z_i$ to be a finite generating set for $M_i$  
containing $Y_j$ for all $j\ne i$.

This provides an MCS $(Y_i;Z_i)$ for $H$: each of the properties MCS(1-4) is explicit in the
construction, as are
MCS$(5)'$ and MCS(6).

It remains to construct an MCS  for $S$ from the one just
constructed for $H=S/Z(S)$. We know from Lemma \ref{l:Z}
that $Z(S)$ is a finitely
generated free abelian group.  
To obtain an MCS $(\hat Y_i; \hat Z_i)$ for $S$, we lift each $Y_i\subset H$ to a finite 
subset $\hat{Y}_i$
of $S$, and take a finite subset $\hat{Z}_i$ in the
preimage of each $Z_i$ 
containing (i) $\hat{Y}_j$ for all
$j\ne i$, and (ii) a finite generating set for $Z(S)$.

To see that $(\hat Y_i; \hat Z_i)$ satisfies
MCS(1), note that $Z(S)\cap [S,S]=1$. Modulo this observation,
it is clear that $(\hat Y_i; \hat Z_i)$
inherits the properties MCS(1-6) from $(Y_i;Z_i)$. 
\end{pf}

\subsection{Two useful lemmata}

The following are the two principal lemmata used in the proof of Theorem 
\ref{t:ee(S)}. 
We first prove a technical lemma about splittings
which allows us to detect when a given quotient of  
$S$ is a non-abelian limit group
rather than a direct product.

 \begin{lemma}\label{jrl}
Let $\G$ be a torsion--free group, $H$ a group,
and $G\hookrightarrow\G\times H$ a subdirect product
such that $G\cap\G$ contains a free group of rank 2.
Let $N$ be a normal subgroup of $G$ with $N<K=G\cap H$.
If $G/N$ admits a  cyclic splitting and $N\neq K$, 
then $K/N$ is cyclic and the splitting is
over $K/N$.
\end{lemma}

\begin{proof} The quotient $G/N \hookrightarrow \G\times H/N$ is a subdirect
product.

The cyclic splitting gives a $G/N$ action on a tree $T$
which is  edge-transitive and has cyclic edge-stabilisers.
A free subgroup $F = \<x,y\>$ of $G \cap \G$ either fixes a vertex
$v$ or contains an element $w$ acting hyperbolically 
(with axis $A$, say).
In the first case $v$ is unique (since $F$ cannot fix an edge), 
so $v$ is
$K/N$-invariant since $K/N$ commutes with $F$.  
But $K/N$ is normal so $K/N$ also fixes $g(v)$ for all $g\in G$.  
Pick $g$ with $g(v) \ne v$, then
$K/N$ fixes more than one vertex, and hence fixes an edge.

In the second case, the axis $A$ is $K/N$-invariant since $K/N$ commutes with $w$.  
If the action of $K/N$ on $A$ is non-trivial, then $A$ is the
(unique) minimal $K/N$--invariant subtree of $T$. 
But then $T$ is $F$-invariant since $F$ commutes with $K/N$.
Thus $F$ acts non--trivially
on $A$ with cyclic edge-stabilisers, which is impossible. 
Hence $K/N$ fixes an edge.

In both cases, $K/N$ fixes an edge, 
hence fixes all edges since $K/N$ is
normal and the action is edge-transitive.  
Thus  $K/N$ is a cyclic group acting trivially on $T$. Moreover, since $K/N\neq 1$,
it has finite index in every (cyclic) edge stabiliser.
Therefore, the action of  $\G = G/K$ on $T$ has finite
cyclic edge stabilisers of the form $Stab_G(e)/K$.  
But $\G$ is torsion-free so these stabilisers are all trivial.
\end{proof}

\begin{lemma}\label{mcs-suffices}
Suppose $S$ is a finitely presented residually free group and
 that  $(Y_1,\ldots,Y_n; Z_1,\ldots,Z_n)$ 
is an MCS for $S$.  Then:
\begin{enumerate}
\item[(0)] each of the groups $S_i/\langle Z_i\rangle$ is a non-abelian
limit group;
\item[(1)]  the natural homomorphism 
$S\to S/\sg{Z_1}\times\cdots\times S/\sg{Z_n}$  
has kernel $Z(S)$ and so embeds $S/Z(S)$ as a full subdirect product
of $n$ non-abelian limit groups;
\item[(2)]   the natural homomorphism 
$S\to \ab{\G}\times S/\sg{Z_1}\times\cdots\times S/\sg{Z_n}$
is an embedding, where $\ab{\G}=\go{S}$.
\end{enumerate}
\end{lemma}

\begin{definition} \label{d:env} To obtain the {\em{reduced existential envelope}} of $S$ 
we fix an MCS $(Y_1,\ldots,Y_n; Z_1,\ldots,Z_n)$
and define $\eer{S}:=S/\sg{Z_1}\times\cdots\times S/\sg{Z_n}$. The
existential envelope of $S$ is then defined  to be 
$\ee{S}= \ab{\G}\times \eer{S}$, where
$\ab{\G}=\go{S}$.
\end{definition}
\begin{remark}
The above definition makes sense in the light of
Lemma \ref{mcs-suffices} and Lemma \ref{mcs-exists}. In the proof of
Lemma \ref{mcs-exists}, we chose the $Z_i$ so that
$M_i=\<Z_i\>$, in the notation of Proposition \ref{charsubgp}, and
we shall see in a moment that this equality is forced by the definition of
an MCS alone. 
The canonical nature of the $M_i$ makes envelopes 
more canonical than they appear in the definition --- Theorem \ref{t:ee(S)}(4-5)
makes this assertion precise. 
\end{remark}

\smallskip

\noindent{\em{Proof of Lemma \ref{mcs-suffices}}}
Suppose that
$(Y_1,\ldots,Y_n; Z_1,\ldots,Z_n)$ is an MCS for the finitely
presented residually free group $S$.
Then by MCS(3) we know $\sg{Z_i}\subseteq C_S(Y_i)$.  
Now there are $x_i,y_i\in Y_i$ such that $[x_i,y_i]\neq_S 1$.
Moreover $[x_i,y_i]\notin C_S(Y_i)$ because $S$ is residually free. 
Hence the images of $x_i$ and $y_i$ in $S/\sg{Z_i}$ form a non-commuting
pair.  Writing $S$ as a subdirect product of some collection
$\G_1,\dots,\G_n$ of limit groups, the projections of $x_i$ and $y_i$
into one of the factors $\G_j$, say, do not commute.
Now we see that $S$ is a subdirect product of $\G\times H$, where $\G=\G_j$ is
a non-abelian limit group, $H$ is a subdirect product of the $\G_i$ ($i\ne j$),
 and $Z_i\subset H$ (by commutative transitivity in $\G$).

Now put $N=\sg{Z_i}\triangleleft S$ (by MCS(4)), and note that $N\subset K:=S\cap H$.
It follows from MCS(5) that $S/N$  
admits a splitting
either over the trivial subgroup or a non-normal, infinite cyclic
subgroup.
Then by Lemma \ref{jrl}, if $K\ne N$, then the splitting is over $K/N$ -
a contradiction since $K/N$ is normal in $S/N$.

Hence  $\sg{Z_i}=N=K=S\cap H$, so
$S/\sg{Z_i}\cong\G$ is a non-abelian limit group, which proves (0).

Since limit groups are fully residually free, 
the centralizer of any non-commuting pair of elements in $S/\sg{Z_i}$ is trivial. Thus $\sg{Z_i}$
is maximal among the centralizers of non-commuting pairs of elements of $S$
(cf.~Proposition \ref{charsubgp}).
In particular $\sg{Z_i} = C_S(Y_i)$  and $\nc{Y_i}\subseteq C_S(\sg{Z_i})$.
Clearly each $\sg{Z_i}\supseteq Z(S)$.

Suppose now that $1\neq  u\in \sg{Z_1}\cap\cdots\cap \sg{Z_n}$ but $u\notin Z(S)$.
Then there is some other element $v$ with $[u,v]\neq 1$.
Since $S$ is residually free, $u$ and $v$ freely generate a free subgroup 
of rank 2.  Thus $u$ and $v^{-1}uv$ freely generate a free subgroup of
$\sg{Z_1}\cap\cdots\cap \sg{Z_n}$ which centralizes each $\nc{Y_i}$.  
So their images in $S/\nc{Y_1,\ldots,Y_n}$ freely generate a free 
subgroup which contradicts MCS(6).
Thus $\sg{Z_1}\cap\cdots\cap \sg{Z_n} = Z(S)$.  This proves (1).

The existence of the embedding in (2) follows immediately from (1), in
the light of Lemma \ref{l:Z}.
\qed

\subsection{Proofs of Theorem \ref{t:ee(S)}(1) and \ref{t:ee(S)}(2)}\label{ss:A12}
   
We are given a finite presentation $\langle A\mid R\rangle$ for a residually free group
$S$. In order to prove Theorem \ref{t:A1}, we must describe an
algorithm that will construct an MCS for $S$ from this presentation:
we know by Lemma \ref{mcs-exists} that $S$ has an MCS and
we know from Lemma \ref{mcs-suffices} (and Definition \ref{d:env})
how to embed $S$ in its envelopes once an MCS is constructed. 

We shall repeatedly use the fact that one can use the given presentation
of $S$ to solve the word problem explicitly: one enumerates homomorphisms
from $S$ to the free group of rank 2 by choosing putative images for the
generators $a\in A$, checking that each of the relations $r\in R$ is mapped
to a word that freely reduces to the empty word; if a word $w$ in the
letters $A^{\pm 1}$ is non-trivial is $S$, one will be able to see this in one of the
free quotients enumerated, since $S$ is residually free. (Implementing a
naive search that verifies if
$w$ does equal the identity is a triviality in any recursively presented group.)

Using this solution to the word problem, we can recursively enumerate
all finite collections 
$\Delta = (Y_1,\ldots,Y_n; Z_1,\ldots,Z_n)$ of finite subsets
of $S$ satisfying conditions MCS(1), MCS(2) and MCS(3).  Next we enumerate all equations
in $S$ and look for those of the form $a^{-1}z a=_S w(Z_i)$  where $z\in Z_i$
and $a^{\pm 1}$ is a  generator of $S$ (and $w$ any word on $Z_i$). 
If a given $\Delta$ satisfies MCS(4), we will eventually discover this by checking the
list of equations. (As ever with such processes, one runs through the finite diagonals
of an array, checking all equations against all choices of $\Delta$.)
Thus we obtain an enumeration of those $\Delta$ satisfying
MCS(1-4). 

Next, we must describe a process that,
given $$\Delta= (Y_1,\ldots,Y_n; Z_1,\ldots,Z_n),$$ can determine if it satisfies MCS(5),
i.e.~if each of the groups $S/\langle Z_i\rangle$ has a splitting of the
required form. Again we only need a
 process that will terminate if $\Delta$ does indeed satisfy MCS(5) --- we are content
for it not to terminate if MCS(5) is not satisfied. 

We have a finite presentation $\langle A\mid R, Z_i\rangle$  for $S/\langle Z_i\rangle$.
By applying Tietze moves (or searching naively for inverse pairs of isomorphisms) we can
enumerate finite presentations of $S/\langle Z_i\rangle$ that have one of the following two  forms
$$
\langle A_1, A_2 \mid R_1, R_2, u_1u_2\rangle,\ \ 
\langle A_1,t \mid R_1, \, tu_1t^{-1}v\rangle,
$$
where $A_1,\, A_2$ and $\{t\}$ are disjoint sets, $R_i\cup\{u_i\}$ is a set of words in the letters $A_i^{\pm 1}$,
and $v$ is a word in the letters $A_1^{\pm 1}$ .
These are the standard forms of presentation for groups that split over (possibly trivial
or finite) cyclic groups. When we find such a presentation, we can use the solution to
the word problem in $S$ to determine if at least one of the generators from $A_1$
and (for the first form) one from $A_2$ are non-trivial in $S$. We proceed to the
next stage of the argument only if non-trivial elements are found.
In the next stage, we use the solution to the word problem to check if $u_1=u_2=1$ in $S$
(or $u_1=v=1$). If these equalities hold, we have found the desired splitting over the
trivial group. If not, then we have a splitting over a non-trivial cyclic group, and since $S$ is
torsion-free, this cyclic group $C=\langle u_1\rangle$
must be infinite. In a residually free group, each 2-generator
subgroup is free of rank 1 or 2 (consider the image of $[x,y]$ in a free group). 
Thus $C$ is normal if and only if it is central, and this can be determined by applying
the solution of the word problem to all commutators $[u,a]$ with $a\in A_1\cup A_2$
(resp. $a\in A_1$). In the case of amalgamated free products, we require
that there be a generator in each of $A_1$ and $A_2$ that does not commute with $C$,
in order that the splitting be non-degenerate.
This concludes the description of the process that will correctly determine if a given
 $\Delta= (Y_1,\ldots,Y_n; Z_1,\ldots,Z_n)$ satisfies MCS(5), halting if it does (but
not necessarily halting if it does not).

Finally, we use coset enumeration to get presentations $\langle A'\mid R'\rangle$
of subgroups of finite
index $S_0\subset S$ with $Y_i\subset S_0$, and we enumerate equations in
the quotients $\langle A'\mid R', Y_1,\dots,Y_n\rangle$ 
to see if the generators satisfy the defining relations of 
the free nilpotent group of class $n-2$ on $|A'|$ generators (and we need only
look for a positive answer). 
As an MCS for $S$ exists (Lemma \ref{mcs-exists}) this 
process will eventually terminate, yielding an explicit $\Delta$ satisfying MCS(1-6).

Part  (2) of Theorem \ref{t:ee(S)} follows immediately from part 1 in the light
of Proposition \ref{p:nilp}.
\qed

\subsection{Proof of Theorem \ref{t:ee(S)}(3) [the universal property of $\eer{S}$]}\label{EmbFpRfGp}

We first record the following 
general result which is also used implicitly in our discussion of 
how $\ee{S}$ is related to the Makanin-Razborov diagram
of $S$.

\begin{prop}\label{factorisation}
Let $G$ be a subdirect product of a finite collection of 
groups:
$G<G_1\times\cdots\times G_n$.  Then any homomorphism from $G$ onto
a non-abelian limit group $\G$ factors through one of the projection
maps $p_i:G\to G_i$ ($i=1,\dots,n$).
\end{prop}

\begin{proof}
An easy induction reduces us to the case where $n=2$.

Define $L_i:=G\cap G_i$ for $i=1,2$.  Then $L_i$ is normal in $G$
for each $i$.   Suppose that $\G$ is a non-abelian limit group and
$\phi:G\to\G$ is an epimorphism.   Then $\phi(L_1)$ and $\phi(L_2)$
are mutually commuting normal subgroups of $\phi(G)=\G$.  If (say)
$\phi(L_1)$ is non-trivial in $\G$, then commutative transitivity in
$\G$ implies that $\phi(L_2)$ is abelian.  But $\G$ has no non-trivial
abelian normal subgroups, so $\phi(L_2)$ is trivial.

Hence one or both of $\phi(L_i)$ ($i=1,2$) is trivial.  But if $\phi(L_1)$
is trivial, then $\phi$ factors through $p_2$, while  if $\phi(L_2)$
is trivial, then $\phi$ factors through $p_1$. 
\end{proof}

To prove Theorem \ref{t:ee(S)}(3), let $S$ be a finitely presented, non-abelian, residually free group 
with MCS $(Y_1,\dots,Y_n;Z_1,\dots,Z_n)$. We have
$\rho:S\to\eer{S}=S/\sg{Z_1}\times\cdots\times S/\sg{Z_n}$,
and we are given a homomorphism $\phi:S\to D=\Lambda_1\times\cdots\times\Lambda_m$ with
the $\Lambda_i$ non-abelian limit groups and $\phi(S)$ subdirect.
We must prove that there is a unique
homomorphism $\hat\phi:\eer{S}\to D$ with $\hat\phi\circ\rho=\phi$. 

For $k=1,\dots,m$ let $\phi_k$ denote the composition of
$\phi$ with the projection $D\to \Lambda_k$.  Since $\Lambda_k$ is a non-abelian
limit group, Proposition \ref{factorisation} says that the surjective map $\phi_k:S\to\Lambda_k$
factors through the projection  $S\to S/\<Z_i\>$ for some $i$.
In particular, $\phi_k(Y_j)=1$ for each $j\ne i$, since $Y_j\subset Z_i$.
However, we must have $\phi_k(Y_i)\neq \{1\}$
by MCS(6) (else $\Lambda$ is virtually nilpotent).
Thus $i=i(k)$ is uniquely determined by $k$.

Applying the above in turn to each $\phi_k$ yields a unique $i(k)$ such 
that $\phi_k$ factors through a map $\zeta_k: S/\<Z_{i(k)}\>\to \Lambda_k$.
Putting all these maps together produces the required 
$\hat\phi: \eer{S}\to\Lambda_1\times\dots\times\Lambda_m$.
\hfill $\square$

\subsection{Proof of Theorem \ref{t:ee(S)}(4) [the uniqueness of $\eer{S}$]}

We are assuming that $\phi:S\hookrightarrow D=\Lambda_1\times\dots\times\Lambda_m$
is a full subdirect product of non-abelian limit groups, and we must prove
that $\hat\phi :\eer{S}\to D$ is an isomorphism.

As in the proof of Lemma \ref{mcs-exists}, we can construct an MCS  for
$S$ from the embedding $\phi:S\hookrightarrow D$,
say $(Y_1',\dots,Y_m';Z_1',
\dots,Z_m')$. Here, $Y_i'\subset S$ generates $\phi(S)\cap \Lambda_i$
as a normal subgroup, $Z_i'$ generates the centralizer of $Y_i'$ in $S$, and
$\phi$ induces an isomorphism
$\overline\phi_i: S/\langle Z_i'\rangle\to \Lambda_i$ for $i=1,\dots,m$.

By using $(Y_i';Z_i')$ in place of $(Y_i;Z_i)$ in Definition \ref{d:env}
we obtain an alternative model $\eer{S}' = S/\langle Z_1'\rangle
\times\dots\times S/\langle Z_m'\rangle$ for $\eer{S}$,
and we have an isomorphism 
$\Phi=(\overline\phi_1,\dots,
\overline\phi_m) :  \eer{S}'\to D$ that restricts to $\phi$ on the
canonical image of $S$ in $\eer{S}'$.

In proving
Theorem \ref{t:ee(S)}(3)  we 
established the universal property for $\eer{S}' $. We apply this
to obtain a unique
homomorphism $\alpha : \eer{S}'\to \eer{S}$ extending 
the inclusion $S\hookrightarrow  \eer{S}$. 
Thus we obtain a homomorphism $\alpha\circ\Phi^{-1}:D
\to \eer{S}$ such that $\alpha\circ\Phi^{-1}\circ\phi $ is the
identity on $S$. But this means that 
$\alpha\circ\Phi^{-1}\circ\hat\phi :\eer{S}\to\eer{S}$ extends
${\rm{id}} : S\to S$. The identity map of $\eer{S}$ is also such
an extension, so
by the uniqueness assertion in \ref{t:ee(S)}(3) we have
that $\alpha\circ\Phi^{-1}$ is a left-inverse to $\tilde\phi$. 
By reversing the roles of $\eer{S}$ and $\eer{S}'$ we see that 
it is also a right-inverse.
\hfill $\square$

\medskip

\subsection{Makanin-Razborov Diagrams.}

We explain how existential envelopes are related to Makanin-Razborov diagrams.

The {\em Makanin-Razborov diagram} (or MR diagram)
of a finitely generated group
$G$ is a method of encoding the collection of all epimorphisms from
$G$ to free groups.  The name arises from the fact that these diagrams 
originate from the fundamental work of Makanin \cite{Mak}
and later Razborov \cite{R}
on the solution sets of systems of equations in free groups.

The MR diagram of $G$ consists of a finite rooted tree, where the root
is labelled by $G$ and the other vertices are labelled by limit groups,
with the leaves being labelled by free groups.
The edges are labelled by proper epimorphisms -- the epimorphism labelling
$e=(u,v)$ mapping the group labeling $u$ onto the group labelling $v$.

The basic property of this diagram is that each epimorphism from
$G$ onto a free group can be described using a directed path in this
graph from the root to some leaf, the epimorphism in question being
a composite of all the labelling epimorphisms of edges on this path,
interspersed with suitable choices of `modular' automorphisms of the 
intermediate limit groups that label the vertices.  Details can be
found  in \cite[Section 7]{Se1} and, in different language, \cite[Section 8]{KMeffective}.

An immediate observation is that any epimorphism from $G$ onto a free
group factors through the canonical quotient $G/\FR(G)$, where $\FR(G)$
is the {\em free residual} of $G$, namely the intersection
of the kernels of all epimorphisms from $G$ to free groups.
Thus the MR diagrams of $G$ and of $G/\FR(G)$ are identical.

Observe that $\FR(G/\FR(G))=1$; in other words $G/\FR(G)$
is {\em residually free}.  Thus, when studying MR diagrams for
finitely generated groups, it is sufficient to restrict attention
to the case of residually free groups.

For finitely generated residually free $G$, 
the top layer of the Makanin-Razborov diagram consists of the
set of maximal limit-group quotients of $G$. These are
the factors of our existential envelope $\ee{G}$, namely the 
maximal free abelian quotient $\G_{\rm ab}(G)$ and the non-abelian
quotients $\G_1,\dots,\G_n$.  The fact that one can construct this effectively
was proved by Kharlampovich and Myasnikov in \cite[Corollary 3.3]{KMeffective}.
Indeed, their construction will
construct for {\em any} finitely presented $G$, the embedding of $G/FR(G)$
into its envelope. Our construction of the
embedding $G\hookrightarrow \ee{G}$ is of a quite different nature, and it
works only when $G$ is residually free.  Nevertheless
we feel that there is considerable benefit in its explicit description.
It is also worth noting that neither the construction of our algorithm nor
the proof that it terminates relies on the original results of Makanin and Razborov.

\section{Decision problems}\label{s:decide}

Theorem \ref{t:ee(S)} provides considerable effective
control over the finitely presented  
residually free groups. In this section we use this effectiveness to solve the
multiple conjugacy problem for these groups and the 
membership problem for their finitely presented subgroups. 
Both of these problems are unsolvable in the finitely
generated case, indeed there exist finitely generated
subgroups of a direct product of two free groups for which
the conjugacy and membership problems are unsolvable \cite{cfm-thesis}.

\subsection{The conjugacy problem}

Instead of considering the conjugacy problem for individual
elements, we consider the multiple conjugacy problem, since the proof
that this is solvable is no harder. The multiple conjugacy problem
for a finitely generated group $G$ asks if there is an algorithm
that, given an integer $l$ and
two $l$-tuples of elements of $G$ (as words in the generators), say $x=(x_1,\dots,x_l)$ and
$y=(y_1,\dots,y_l)$, can determine if there exists $g\in G$ such
that $gx_ig^{-1}=y_i$ in $G$, for $i=1,\dots,l$. There exist groups in which
the conjugacy problem is solvable but the multiple conjugacy problem
is not \cite{BH}.

The scheme of our solution to the conjugacy problem uses an
argument from \cite{BM} that is based on
Theorem 3.1 of \cite{mb-haef}. 
This is phrased in terms of
bicombable groups. Recall that a group $G$ with finite generating
set $A$ is said to be {\em bicombable} if there is a constant
$K$ and choice
of words $\{\sigma(g) \mid g\in G\}$ in the letters $A^{\pm 1}$ such
that
$$
d(a.\sigma(a^{-1}ga')_t, \,\sigma(g)_t) \le K
$$
for all $a,a'\in A$ and $g\in G$, where $w_t$ denotes the image in $G$ of
the prefix of length $t$ in $w$, and $d$ is the word metric associated to $A$. 

We shall only use three facts about bicombable groups. First, the
fundamental groups of compact non-positively curved spaces
are the prototypical bicombable groups, and limit groups
are such fundamental groups \cite{AB}. Secondly, there
is  an algorithm that given any finite set $X\subset \G$
as words in the generators of $G$ will
calculate a finite generating set for the centralizer of $X$. (This is 
proved in \cite{mb-haef} using an argument from \cite{GS}.)
Finally, we need the fact that the multiple conjugacy
problem is solvable in bicombable groups. The proof of this is a mild
variation on the standard proof that bicombable groups have a
solvable conjugacy problem. The key point to observe is that,
given words $u$ and $v$ in the generators, if $g\in G$ is
such that $g^{-1}ug=v$, then as $t$ varies, the distance from $1$
to $\sigma(g)_t^{-1}u\sigma(g)_t$ never exceeds $K\max\{|u|,|v|\}$.
It follows that in order to check if two $(u_1,\dots,u_k)$
and $(v_1,\dots,v_k)$  are conjugate
in $G$, one need only check if they are conjugated by an element $g$
with $d(1,g)\le |2A|^{K\max\{|u_i|,\,|v_i|\}}$
(cf.~Algorithm 1.11 on p.~446 of \cite{BHa}).

\begin{proposition}\label{solvC} 
Let $\G$ be a bicombable group, let $H\subset \G$ be a 
subgroup, and suppose that there exists
a subgroup $L\subset H$ normal in $\G$ such that 
$\G/L$ is nilpotent. Then $H$ has a solvable multiple conjugacy
problem.
\end{proposition}

\begin{proof}
Given a positive integer $l$
and two $l$-tuples $\underline x,\underline
y$ from $H$ (as lists of words in the generators of $\G$)
we use the positive solution to the multiple conjugacy problem in
$\G$ to determine if there exists $\gamma\in\G$ such that 
$\g x_i\g^{-1}=y_i$ for $i=1,\dots,l$. If no
such $\g$ exists, we stop and declare that $\underline x$
and $\underline y$ are not conjugate in $H$.
If $\g$ does exist then we find it and consider
$$\g C=\{g\in \Gamma\mid gx_ig^{-1}=y_i\text{ for }i=1,\dots,l\},$$
where $C$ is the centralizer of $\underline x$ in $\G$.
Note that $\underline x$ is conjugate to $\underline y$ in $H$ if and only
if $\g C\cap H$ is non-empty.

We noted above that there is an
algorithm that computes a finite generating
set   for $C$. This enables us to employ Lo's algorithm
(Lemma \ref{Lo}) in the nilpotent group $\G/L$ 
to determine if the image of $\gamma C$ intersects the
image of $H$. Since $L\subset H$, this intersection is
non-trivial (and hence $x$ is conjugate to $y$) if and
only if $\g C\cap H$ is non-empty. 
\end{proof}

A group $G$ is said to have {\em unique roots} if for all
$x,y\in G$ and $n\neq 0$ one has $x=y\ \iff\  x^n=y^n$.
It is easy to see that residually free groups have this property.
As in Lemma 5.3 of \cite{BM} we have:

\begin{lemma}\label{findex} Suppose $G$ is a group in which roots are
unique and  $H\subset G$ is a subgroup of finite index.
If the multiple conjugacy problem for $H$  is solvable, then
the multiple conjugacy problem for $G$ is solvable.
\end{lemma}

The final lemma that we need can be proved by
a straightforward induction on the
nilpotency class, but there is a more elegant argument due to 
Lo (Algorithm 6.1 of \cite{Lo})  that provides an algorithm which is practical
for computer implementation.

\begin{lemma}\label{Lo} If $Q$ is a finitely generated
nilpotent group, then there is an
algorithm that, given finite sets $S,T\subset Q$ and $q\in Q$,
will decide if $q\<S\>$ intersects $\<T\>$ non-trivially. \qed
\end{lemma}

\begin{thm}[=Theorem \ref{t:conj}]\label{t:Mconj} The multiple conjugacy problem is
solvable in every finitely presented residually
free group.
\end{thm}

\begin{proof}
Let $\G$ be a finitely presented residually free group. 
Theorem \ref{t:ee(S)} allows us to
embed $\G$ as a subdirect product in $D=\L_1\times\dots\times\L_n$,
where $\L_i$ are limit groups, each $L_i=\L_i\cap \G$ is non-trivial,
$L=L_1\times\dots\times L_n$
is normal in $D$, and $D/L$ is virtually nilpotent. Let
$N$ be a nilpotent subgroup of finite index in $D/L$,
let $D_0$ be its inverse image in $D$ and let $\G_0=D_0\cap \G$.

We are now in the situation of Proposition \ref{solvC} with
$\G=D_0$ and $H=\G_0$. Thus $\G_0$ has a solvable multiple
conjugacy problem.
 Lemma \ref{findex} applies to residually free groups, 
 so the multiple conjugacy problem for $\G$ is also solvable. 
 \end{proof}

\subsection{The finite presentation problem}

We will need the
following technical observation. This was first proved in \cite{KMeffective}, Theorem 3.21.
It admits a short proof based on Wilton's theorem that finitely generated
subgroups of limit groups are virtual retracts (see \cite{GW} Theorem 2.4).

\begin{lemma}\label{l:makeP} There is an algorithm that, given a finite presentation of a limit group $\L$
and a finite set $X\subset\L$, will
output a finite presentation for the subgroup generated by $X$.
\end{lemma}

Unlike limit groups, finitely generated subgroups of a finitely-presented residually-free group need
not be finitely presentable.  Our next result says that if such a finitely generated 
subgroup is finitely presentable, then  we can  effectively find a presentation.

\begin{theorem}[=Theorem \ref{lt:makeFP}]\label{t:makeFP} 
There is a uniform partial algorithm  for finding presentations of
finitely presentable subgroups of  finitely-presented residually-free groups. 
More precisely, there is a partial algorithm that, given a finite presentation for a 
residually free group $G$ and a finite set of words generating a subgroup $H$, 
will output a finite presentation for $H$ if it exists.
\end{theorem}

To prove this theorem we begin by applying the algorithm of  Theorem \ref{t:ee(S)} 
to find the existential envelope $D$ of $G$  and the images of the generators 
of $H$ in $D$.  So it suffices to consider the case in which $H$ is given by a finite
set of generators in a specified direct product of limit groups.  The theorem then
follows from Proposition \ref{p:makeFP} below.

\begin{remark}  We pause to record the following observation. 
The algorithm in \cite{GW} provides an enumeration of
non-abelian limit groups.  An obvious modification of this enumeration 
produces a recursively enumerable sequence of 
finite presentations for direct products $\G_1\times\dots
\times\G_n$ of limit groups.  
So if we are given any finite presentation of a group  $D$
and told that $D$  is a direct product limit groups, then a naive 
search identifies a presentation for $D$ on this list. 
Thus we may effectively replace the given presentation for $D$ with one in which such a 
direct product decomposition and the coordinate projections 
$p_i:D\to \G_i$ and $q_i:D\to D/\G_i$ are manifest. 
\end{remark}

\begin{lemma}\label{projectFP}
Suppose that $H$ is a subgroup of a direct product $D = \G_1\times \cdots \times  \G_n$
of limit groups $\G_i$.   If $H$ is finitely presented and 
$$\rho: D \to  \G_{i_1}\times \cdots \times \G_{i_k}  \mbox{ where } i_1 < \cdots < i_k$$
is the projection onto 
the product of
any subset of $k$ of the factors,
then $\rho(H)$ is finitely presented.
\end{lemma}

\begin{proof}  If $p_i$ denotes the projection onto the factor $\G_i$, then each $p_i(H)$
is a finitely generated subgroup of a limit group and so is again a limit group.  Hence we 
may assume from the outset that $\G_i =p_i(H)$, that is, $H$ is a subdirect product.
Then every such $\rho(H)$ is also subdirect in the direct product $\rho(D)$.

If $n=1$ or $n=k$ there is nothing to prove.  Suppose that $H\cap \G_j =1$ 
 where  $ \G_j \subseteq \ker \rho$.  Then the  projection 
$$q_j: D \to \G_1\times \cdots \times \G_{j-1}\times \G_{j+1} \times \cdots \times \G_n$$
is injective on $H$.  So by induction on $n$, it follows that $\rho(H) = \rho( q_j(H))$
is finitely presented as required.

Next suppose that $\rho(H) \cap  \G_{i_j} = 1$.  Then $H\cap  \G_{i_j} = 1$ and  the projection 
$$q_{i_j} : D \to \G_1\times \cdots \times \G_{i_j-1}\times \G_{i_j+1} \times 
\cdots \times \G_n$$
is injective on both $H$ and $\rho(H)$.  Now $q_{i_j}(D)$ is a product of $n-1$
limit groups and $q_{i_j}\circ \rho = \rho\circ q_{i_j}$ is projection onto $k-1$ factors, and hence
by induction $q_{i_j}(\rho(H))$ is finitely presented.  But since $\rho(H) \cap \G_{i_j} = 1$,
it follows that $\rho(H)$ is finitely presented.

So we may now assume that both $H$ and $\rho(H)$ are full subdirect products.  Next
we observe generally that if $K$ is a full subdirect product of $D$ and $\G_j$ is an
abelian limit group, then $q_j(K) \subseteq D/\G_j$ is finitely presented if and only 
if $K$ is finitely presented, because $\G_j$ is a finitely-generated central subgroup.
Hence it suffices to consider the case all of the $\G_i$ are non-abelian.

But a full subdirect product of non-abelian limit groups is finitely presented if and only if it
satisfies VSP.  Since $H$ has VSP,  either $\rho(H)$ is a limit group or it also has VSP.
Hence $\rho(H)$ is finitely presented as desired.  
\end{proof}

Making use of this lemma we can now complete the proof of Theorem \ref{t:makeFP}
by showing the following.

\begin{proposition}\label{p:makeFP} 
There is a  partial algorithm that, given a direct product $D$ of limit groups presented as
$D = \G_1\times \cdots\times \G_n$ and a finite set of words $X$ generating a subgroup $H<D$,
\begin{enumerate}
\item\label{OKpres} in case $H$ is finitely presentable, will output a finite presentation 
for $H$ on the given generating set $X$;  or 
\item in case $H$ is not finitely presentable, will either halt saying that $H$ is not finitely presentable,
or will fail to halt.
\end{enumerate}
Moreover, in the case $H$ is finitely presentable, the algorithm will also
\begin{itemize}
\item determine for each $i$ whether or not $H\cap \G_i=1$; and
\item determine a finite generating set for the centre of $H$.
\end{itemize}
\end{proposition}

\begin{proof}  Note that there is a uniform solution to the word problem 
for all such $\G_i$ and $D$ since
they are residually finite and given by finite presentations.  Hence, as a finitely generated
subgroup of $D$,  $H$ also has a solvable word problem.

If $n=1$ then  we apply the algorithm of Lemma \ref{l:makeP}  (in this case
$H$ is finitely presentable).  We now proceed by induction on $n$.  
Let $q_j$ be the projection of $D$ onto all the factors other than $\G_j$, so $q_j:D\to D/\G_j$.
By Lemma \ref{projectFP}, if $H$ is finitely presented then each $Q_j = q_j(H)$
is also finitely presented and lies in a direct product of fewer limit groups, so we have 
(by our inductive assumption)  a partial algorithm to find a presentation for $Q_j$. 

So we now launch $n$ versions of this partial algorithm 
attempting to find a presentation for each $q_i(H)$ in case it is finitely presented ($i=1,\ldots,n)$.
If some process finds a $Q_j$ is not finitely presented, we are done since $H$ cannot
be finitely presented and we halt with this information.

So we may assume each of our $n$ processes finds  that its corresponding 
$Q_i = \<q_i(X)\>$ is finitely presented,  say as $\<x_1,...,x_s \mid  r_1=1,...,r_t=1\>$
where $X=\{x_1,\dots,x_s\}$ and the normal subgroup generated by the $r_j$ is the
kernel of the surjection $x\mapsto q_i(x)$ to $Q$ from the free group on $X$.
Then $q_i$ is an isomorphism from $H$ onto $Q_i$ if and only if all of the $r_j =_H 1$. 
And these equalities can be tested using the solution to the word problem for $D$.  
Of course $H\cap\G_i =1$ if and only if $q_i$ is an isomorphism. 

So if $H$ is finitely presentable, each of these $n$ processes will (eventually) have halted 
and we have decided whether or not each $H\cap\G_i =1$.  If some $H\cap\G_i =1$,
then $Q_i \isom H$ and we have a presentation for $H$ and our inductive algorithm
also provides a set of generators for the centre of $H$ and we are done.

So we may now assume that we have a finite presentation for each  $Q_i$ 
and that $H\cap \G_i \neq 1$, thus $H$ is a full subdirect product.

Observe that using the solution to the word problem we can easily determine
which of the $\G_i$ are abelian (and then they must be free abelian).
So next suppose that none of the $\G_i$ is abelian.  In this case, we know $H$ is 
finitely presentable if and only if it satisfies VSP.   Also, the centre of $H$
is trivial in this case. 

If $n=2$, this is true if and only
if $H$ has finite index in $\G_1\times \G_2$.  So if $n=2$, we start trying to find
a presentation for $H$ using the classic Todd-Coxeter algorithm.  If $H$ is finitely
presentable, it will terminate and give us a presentation for $H$.  
If $H$ is not finitely presentable, the process will fail to terminate.
(Note:  whether this process will terminate is an unsolvable problem.)

If $n>2$, then since each $Q_i$ is  full and finitely presented, each $Q_i$ has VSP
and hence $H$ has VSP.  Thus applying the algorithm of Theorem \ref{t:effFPsubdirect}
we can effectively find a finite presentation $\mathcal{P}$ of $H$ on the generators  $X$. 

Now suppose that at least one of the $\G_i$ is abelian, and recall it is a finitely generated free
abelian group.  Group all the abelian factors together as one factor and revise the notation so that
$\G_n$ is the only abelian direct factor.  Then, applying our algorithm inductively,
we can assume we have a finite presentation for $q_n(H)= Q_n = \<q_n(X)\>$,  
say as $\<x_1,...,x_s \mid  r_1=1,...,r_t=1\>$.  We also know that $q_n(H)\isom H/Z(H)$. 
Lifting the generators of this presentation back up to the corresponding  generators of $H$,
the same words $r_j$, which are relators of $Q_i$,  are elements in $D$ such that $r_j \in \G_n$.
Further,  they generate the centre $Z(H)$ of $H$.  
Using standard algorithms for finitely generated abelian groups,
we can find a set of defining relations, say $z_1=1,\ldots,z_k=1$ 
for the subgroup  of $\G_n$ generated by the $r_i$.
Finally, expressing the $z_j$ as words $\hat{z}_j$ in the $x_i$, 
we can then write down a presentation for $H$ as 
$$H =\< x_1,...,x_s \mid  [x_i, r_j] = 1 \mbox{ for } i=1,\ldots,s,\ j=1,\ldots,t,$$
$$ \hat{z}_1=1,\ldots,\hat{z}_k=1  \>.$$ 
This complete the proof.
\end{proof}

As previously explained this also completes the proof of Theorem \ref{t:makeFP}.

\subsection{The membership problem}
Let $G$ be a recursively presented group,
and $H$ the subgroup generated by a set of words
in the generators of $G$.  The {\em membership
problem} for $H$ in $G$ is the algorithmic problem of
deciding, given a word $w$ in the generators of $G$, whether
or not the element $g\in G$ represented by $w$ belongs to $H$. 

\begin{remark}\label{r:sepbl}
There is an obvious (uniform) algorithm to solve the membership problem for separable subgroups of finitely
presented groups: one runs a naive search to express the given element as a word in the generators of the
subgroup while, in parallel, enumerating the finite quotients of the ambient group, checking to see if the
element is separated from the subgroup in any of them.
\end{remark}

We prove that there is a uniform partial algorithm to solve the membership problem for finitely presentable subgroups 
of finitely presented residually free groups.

\begin{thm}[=Theorem \ref{t:memb}]\label{t:mem1}  
There is a uniform partial algorithm that, 
given a finite presentation of a residually free group $G$,
 a finite generating set for a   subgroup $H\subset G$ and 
 a word  $g$ in the generators defining $G$, 
 will determine whether or not  $g$  lies in $H$, provided that $H$ is finitely presentable.  
 \end{thm} 

\begin{proof} The algorithm given by
Theorem \ref{t:ee(S)} embeds $G$ in a direct
product $\Delta$ of limit groups. A solution
to the membership problem for $H\subset \Delta$ provides a solution
for $H\subset G$. Thus there is no loss of generality in
assuming that $G$ is a direct product of limit groups,
say $G=\L_1\times\dots\times\L_n$. 

To complete the proof,
we argue by induction on $n$. The case $n=1$ is covered by
the fact that limit groups are subgroup separable \cite{wilton}.

We first employ the algorithm of Proposition \ref{p:makeFP}: if $H$ is not finitely
presentable then this algorithm will either fail to halt
or else halt and inform us that $H$ is not finitely
presentable; if $H$ is finitely presentable
then it will halt and list the indices $i$ such that $L_i:=H\cap \L_i =1$.

There is no loss of generality in assuming that elements
$g\in G$ are given as words in the generators of the factors,
and thus we write $g=(g_1,\dots,g_n)$. We assume that the
generators of $H$ are given likewise. 

We first deal with the case where some $L_i$ is trivial, say $L_1$.
The projection of
$H$ to $\L_2\times\dots\times\L_n$ is then isomorphic to $H$, so
in particular it is finitely presented and our induction provides
an algorithm that determines if $(g_2,\dots,g_n)$ lies in
this projection. If it does not, then $g\notin H$. If it does,
then naively enumerating equalities $g^{-1}w=1$ we
eventually find a word $w$ in the generators of $H$ so that
$g^{-1}w$ projects to $1\in \L_2\times\dots\times\L_n$.
Since $L_1=H\cap\L_1=\{1\}$, we deduce that 
in this case $g\in H$ if and only if  $g^{-1}w=1$, and the
validity of this equality can be checked using the uniform solution to the word
in residually free groups.

It remains to consider the case where $H$ intersects
each factor non-trivially. Again we are given $g=(g_1,\dots,g_n)$.
The projection $H_i$ of $H$ to $\L_i$ is finitely generated
and Wilton's theorem \cite{wilton} tells us
that $\L_i$ is subgroup separable, so we can determine
algorithmically if $g_i\in H_i$. If $g_i\notin H_i$ for some $i$
then $g\notin H$ and we stop. Otherwise, we replace $G$
by the direct product $D$ of the $H_i$. Lemma \ref{l:makeP}
allows us to compute a finite presentation for $H_i$
and hence $D$.

We are now reduced to the case where $H$ is a full subdirect
product of $G(=D)$. Such subgroups are separable, by Corollary \ref{c:sdp-sep},
so remark \ref{r:sepbl} completes the proof.
\end{proof}

\begin{remark} Following our work, Bridson and
Wilton \cite{BWilt}  proved that in the profinite topology 
on any finitely generated residually free group, all finitely
presentable subgroups are closed.  Using the results
of \cite{BWilt} and \cite{BHMS1},  Chagas and Zalesski \cite{CZ} proved that 
all finitely presented residually free groups are conjugacy separable.
\end{remark}

\subsection{Recursive enumerability}

In view of the insights we have gained into the structure of finitely
presentable residually free groups, 
it seems reasonable to conjecture that the isomorphism problem for this class of groups is solvable.
We have not yet succeeded in constructing an algorithm
to determine isomorphism, but we are nevertheless able to prove
the following partial result in this direction.

\begin{theorem}[= Theorem \ref{t:re}]\label{t:reFGRF}
The class of finitely presentable residually free groups is
recursively enumerable.   More precisely,
there is a Turing machine that will output a list of finite
group-presentations $\mathcal{P}_1,\mathcal{P}_2,\dots$ such that:
\begin{enumerate}
\item the group $G_i$ presented by each $\mathcal{P}_i$ is residually free; and
\item every finitely presented residually free group is isomorphic to at
least one of the groups $G_i$.
\end{enumerate}
\end{theorem}

\begin{proof}
First we enumerate the limit groups, using the
algorithm in  \cite{GW}.
This leads in a standard way to an
enumeration of finite subsets $Y$ of finite direct products thereof:
$Y\subset D:=\G_1\times\cdots\times\G_n$.

For each such $Y$ and each pair $i,j$, the Todd-Coxeter procedure
will tell us if $p_{ij}(Y)$ generates a finite-index subgroup
of $\G_i\times\G_j$ (but will not terminate if it does not).

Whenever we encounter a finite collection of limit groups
$\G_1,\dots,\G_n$ and a finite subset $Y\subset D$
such that $p_{ij}(Y)$ generates a finite-index subgroup
of $\G_i\times\G_j$ for all $i,j$, we set about constructing
a finite presentation for the subgroup generated by $Y$, using
Theorem \ref{t:effFPsubdirect}.

Thus a list can be constructed of all finitely-presented full
subdirect products of limit groups, together with a finite
presentation for each one.  By Theorem \ref{t:main}
this list contains (at least one isomorphic copy of)
every finitely presentable residually free group.
\end{proof}

The facts we have proved or mentioned
in this paper provide recursive enumerations of various other classes of groups:

\begin{enumerate}
\item\label{i:e1} There is a recursive enumeration of the finitely generated residually free 
groups: each 
is given by a finite set $X$  that
generates a full subdirect product in a finite direct product of limit 
groups $\G_1\times\cdots\times\G_n$.
\item\label{i:e2} One can extract from (1) a recursive enumeration of the finitely generated residually free 
groups with trivial centre (those for which each $\G_i$ is non-abelian), and a complementary enumeration
of those with non-trivial centre.    
\item\label{i:e4} The subsequence of (\ref{i:e1}) consisting of those groups that are finitely presentable 
is recursively enumerable (Theorem \ref{t:reFGRF}).
\item\label{i:e5} The subsequence of (\ref{i:e4}) consisting of those finitely presented residually 
free groups with trivial (resp. non-trivial) centre is recursively enumerable.
\end{enumerate}

\subsection{Partial results on the isomorphism problem}

Suppose we are given two finite presentations of residually free groups $G$ and $H$.  Can we decide algorithmically
whether or not $G\cong H$?

There is a partial algorithm that will search for a mutually
inverse pair of isomorphisms, expressed in terms of the 
given finite generating sets for $G$ and $H$.  This will
terminate if and only if $G\cong H$, giving us the desired
isomorphism in the process.

The difficult part of the problem is therefore to recognise,
via invariants or otherwise, when $G\not\cong H$.

Our earlier results have provided computations of an important
invariant, namely 
the set of maximal limit group quotients
of $G$.  Using the solution to the isomorphism 
problem for limit groups (\cite{BKM,DG}), we can 
distinguish $G$ from $H$ unless these  agree
for $G$ and $H$.  The problem is thus effectively reduced
to the case where $G$ and $H$ are specifically given to us
as full subdirect products of limit groups $\G_1,\dots,\G_n$.

Moreover, by Proposition \ref{p:makeFP} we can effectively
determine whether or not $Z(G)\cong Z(H)$.  So we may assume that
$Z(G)\cong Z(H)=Z$, say, and that the $\G_i$ are all 
non-abelian if $Z$ is trivial.  
In the case where $Z$ is non-trivial, then precisely one of the $\G_i$ is abelian.
We make the convention that in this case $\G_1$ is abelian.
Then $\G_1\cong \go{G} \cong \go{H}$, and $Z(G)=G\cap\G_1$,
$Z(H)=H\cap\G_1$.  Under these circumstances, as a special
case of Theorem \ref{t:ee(S)}(4) we have:

\begin{prop}\label{isos}
Any isomorphism $\theta:G\to H$ is the restriction of
an ambient automorphism of the direct product
$\G_1\times\cdots\times\G_n$.  This in turn restricts to
a set of isomorphisms $\G_i\to\G_{\sigma(i)}\ (i=1,\dots,n)$
for some permutation $\sigma$ of $\{1,\dots,n\}$.
\end{prop}

Since there are only finitely many candidate permutations $\sigma$,
this proposition effectively reduces the isomorphism problem to the 
case where $\sigma$ is the identity, in other words to the following:

\medskip\noindent{\bf Question}: Given finitely presented full subdirect products
$G,H$ of a collection of limit groups $\G_1,\dots,\G_n$ (at most one of which
is abelian), can we find
automorphisms $\theta_i$ of $\G_i$ for each $i$, such that
$$(\theta_1,\dots,\theta_n)(G)=H?$$

Recall that the automorphism groups of limit groups can be effectively 
described \cite{BKM};
in particular one can find finite generating sets $X_i$ for each $Aut(\G_i)$.

\begin{prop}
There is a solution to the isomorphism problem in the case when at most
$2$ of the $\G_i$ are non-abelian.
\end{prop}

\begin{proof}  
By Proposition \ref{p:makeFP} we can effectively find neat embeddings for $G$ and $H$.
Hence there is no loss of generality in assuming that the given embeddings
$G,H\hookrightarrow D:=\G_1\times\cdots\times\G_n$ are neat.  In particular, at most
two of the $\G_i$ are non-abelian, at most one is abelian, and $G,H$ intersect
any abelian direct factor in a subgroup that has finite index in that factor.
The VSP property then ensures that each of $G,H$ has finite index in
$D$.

The index can be computed in each case using the Todd-Coxeter algorithm, and we
may assume that the two indices are equal (to $k$, say).  Now by
\cite{BKM} we can find
a finite set $X=X_1\times\cdots\times X_n$ of generators for $\Theta
=Aut(\G_1)\times\cdots\times Aut(\G_n)$.

It is straightforward to construct the permutation graph for the action
of $\Theta$ on the finite set of index $k$ subgroups, and then to check
whether or not $G$ and $H$ lie in the same component of this graph.  This happens if and only if $G$
is isomorphic to $H$ via an automorphism of $D$ that preserves the direct
factors.  By Proposition \ref{isos}, this suffices to solve the problem.

\end{proof}

One possible approach to the more general case is to proceed
by induction on the number of direct factors.  Projecting a 
finitely presentable subdirect product to the product of fewer
factors again gives a finitely presentable group, so by induction
we can assume that the corresponding projections of our
two subgroups are isomorphic.  But for the moment we do
not see how this information might be used to complete a proof that the
isomorphism problem is solvable.


\begin{thebibliography}{999}

\bibitem{AB}
E.~Alibegovi\'c and M.~Bestvina,
{\em Limit groups are CAT$(0)$},  J. London Math. Soc. (2)
{\bf 74}  (2006), 259--272.


\bibitem{BBMS} G.~Baumslag, M.~R. Bridson, C.~F. Miller III and H. Short,
{\em Fibre Products, non-positive curvature, and decision problems},
Comment. Math. Helv. {\bf 75} (2000), 457-477.
 
\bibitem{BMR}
G. Baumslag, A. Myasnikov, and V. Remeslennikov.
\newblock {\em Algebraic geometry over groups. {I}. {A}lgebraic sets and ideal
  theory.}
\newblock J. Algebra {\bf 219} (1999), 16--79.

\bibitem{BR} G. Baumslag and J. Roseblade, {\em Subgroups of direct products of
free groups}, J. London Math. Soc. (2) {\bf 30} (1984) 44--52.

\bibitem{Bieri} R. Bieri, {\em Homological dimension of discrete groups},
Queen Mary College Mathematics Notes (1976).

\bibitem{mb-haef} M.~R.~Bridson, {\em On the subgroups of semihyperbolic groups},
in ``Essays on geometry and related topics",  pp.~85--111, Monogr. Enseign. Math. {\bf{38}},
Geneva, 2001.

\bibitem{BHa} M.~R.~Bridson and A.~Haefliger,
``Metric Spaces of Non-Positive Curvature", Grund. Math. Wiss.
{\bf 319}, Springer-Verlag, Heidelberg-Berlin, 1999.

\bibitem{BH} M.~R. Bridson and J. Howie, 
{\em Conjugacy of finite subsets in hyperbolic groups}.
Internat. J. Algebra Comput. {\bf 15} (2005), no. 4, 725--756.



\bibitem{BH2} 
M.~R.~Bridson and J.~Howie, {\em Subgroups of direct
products of elementarily free groups}, 
 Geom. Funct. Anal. (GAFA)  {\bf{17}}  (2007),   385--403. 


\bibitem{BHMS} M.~R.~Bridson, J.~Howie, C.~F. Miller~III and H.~Short,
{\em The subgroups of  direct products of surface groups}, 
Geometriae Dedicata {\bf 92 } (2002), 95--103.


\bibitem{BHMS1}
M.~R. Bridson, J. Howie, C.~F. Miller~III and H.~Short,
\newblock {\em Subgroups of direct products of limit groups.}
\newblock Ann. of Math. {\bf 170} (2009), 1447--1467.


\bibitem{BHMS2}
M.~R. Bridson, J. Howie, C.~F. Miller~III and H.~Short,
\newblock {\em Finitely presented residually free groups.}
\newblock ArXiv:0809.3704 (2008).

\bibitem{BM1}
M.~R. Bridson and C.~F. Miller~III,
\newblock {\em Recognition of subgroups of direct products of
hyperbolic groups},  Proc. of Amer. Math. Soc., {\bf 132} (2003), 59--65.

\bibitem{BM}
M.~R. Bridson and C.~F. Miller~III,
\newblock {\em Structure and finiteness properties of subdirect products of groups.}
\newblock Proc. London Math. Soc. (3) {\bf 98} (2009), 631--651.


\bibitem{BWilt} 
M.~R.~Bridson and H.~Wilton, 
{\em Subgroup separability in residually free groups}, Math. Z.
{\bf 260} (2008), 25--30.
 
 
\bibitem{BWilt2} 
M.~R.~Bridson and H.~Wilton, 
{\em On the difficulty of presenting finitely presentable groups}, Groups Geom. Dyn. {\bf 5} (2011), 301--325. 

\bibitem{BKM} I. Bumagin, O. Kharlampovich and A. Miasnikov,
{\em The isomorphism problem for finitely generated fully residually free groups.}
J. Pure Appl. Algebra {\bf 208} (2007), 961--977.


\bibitem{CZ} S.~C. Chagas and P.~A. Zalesskii,
{\em{Finite index subgroups of conjugacy separable groups}}, Forum Math. {\bf 21}
(2009), 347--353.


\bibitem{DG} F. Dahmani and D. Groves,
The isomorphism problem for toral relatively hyperbolic groups,
Publ. Math. Inst. Hautes \'{E}tudes Sci. {\bf 107} (2008), 211--290.

\bibitem{DeGr} T. Delzant and M. Gromov, {\emph{Cuts in K\"{a}hler groups}}, in {\emph{Infinite groups: geometric,
combinatorial and dynamical aspects}}, (L. Bartholdi and others,
eds.), Progr. Math., 248, Birkh\"{a}user, Basel, 2005, pp. 3155. MR 2006j:32023

\bibitem{DPS} A. Dimca, S. Papadima and A. I. Suciu, {\em Non-finiteness
properties of fundamental groups of smooth projective varieties}, J. Reine Angew. Math. {\bf 629} (2009), 405--457.

\bibitem{will} W. Dison, {\em{Isoperimetric
functions for subdirect products and Bestvina-Brady groups}}, Ph.D. thesis, Imperial College London, 2008.

\bibitem{ross} R. Geoghegan,
{\em Topological Methods in Group Theory.}
Graduate Texts in Mathematics vol.~243, Springer-Verlag, 
New York 2008.

\bibitem{GS} S.~M.~Gersten and H.~B.~Short,  
{\em Rational subgroups of biautomatic groups},
Ann. of Math.  {\bf{134}} (1991),  125--158. 

\bibitem{GW} D. Groves and H. Wilton,
{\em Enumerating limit groups},
Groups Geom. Dyn. {\bf 3} (2009), 389--399.

\bibitem{GL}
V. Guirardel and G. Levitt, {\em Computing equations for residually free
groups}, Illinois J. Math. {\bf 54} (2010), 129--135.

  \bibitem{KM2}
O.~G.~Kharlampovich and A.~G.~Myasnikov, 
\emph{Irreducible affine varieties over a
  free group. {II}. {S}ystems in triangular
   quasi-quadratic form and
  description of residually free groups}, 
  J. Algebra \textbf{200} (1998),
  517--570.

\bibitem{KMeffective} O. Kharlampovich and A. Myasnikov, \emph{Effective JSJ
decompositions}, in:\\ \emph{Groups, languages, algorithms} (A. V. Borovik, ed.)  Contemp. Math., 378, Amer. Math. Soc., Providence, RI, 2005, pp 87--212.


\bibitem{desi}  D. Kochloukova, 
{\em{On subdirect products of
type $FP_m$ of limit groups}},  J. Group Theory {\bf 13} (2010), 1--19.

\bibitem{KS} D. Kochloukova and H. Short, {\em On subdirect products of free pro-$p$ groups and Demushkin groups of infinite depth}, J. Algebra {\bf 343} (2011), 160--172.

\bibitem{Lo}
E.~H.~Lo, 
{\em Finding intersections and normalizers in finitely generated nilpotent groups},
J. Symbolic Comput.  {\bf{25}}  (1998),  45--59.

\bibitem{Mag}
W.~Magnus,
{\em \"{U}ber Gruppen und zugeordnete Liesche Ringe},
J. Reine Angew. Math. {\bf 182} (1940), 142--149.

\bibitem{MKS}
W.~Magnus, A. Karrass  and D. Solitar,
``Combinatorial Group Theory", Wiley, New York, 1966.
 

\bibitem{Mak}
G.~S.~Makanin, {\em Equations in a free group}, Math. USSR Izv. {\bf 21} (1983),
483-546.

\bibitem{cfm-thesis}  C.~F.~Miller~III, 
``On group-theoretic decision problems and their classification",
Annals of Mathematics Studies,
No. 68, Princeton University Press (1971).

\bibitem{R}
A.~A.~Razborov, {\em On systems of equations in a free group},
Math. USSR Izv. {\bf 25} (1985), 115-162.

\bibitem{Se1}
Z. Sela,
{\em{Diophantine geometry over groups. {I}. {M}akanin-{R}azborov diagrams}},
Publ. Math. Inst. Hautes \'Etudes Sci., pages 31--105, 2001.

\bibitem{St} J. R. Stallings, {\em A finitely generated group whose 3--dimensional
homology group is not finitely generated}, Amer. J. Math., {\bf 85} (1963), 541--543. 
\bibitem{whead}
J.~H.~C.~Whitehead, 
{\em{On Adding Relations to Homotopy Groups}},
Ann. of Math. {\bf{42}} (1941),  409--428.


\bibitem{wilton}
H. Wilton.
\newblock Hall's {T}heorem for limit groups,
\newblock \emph{Geom. Funct. Anal.} {\bf 18} (2008),
25--30.

\end{thebibliography}
\end{document}